\documentclass[11pt,a4paper,reqno]{amsart}
\usepackage{amsmath,amsthm,amsfonts,amssymb,bm,wasysym} 
\usepackage{amsaddr}
\usepackage{comment} 
\usepackage{caption} 
\usepackage{subcaption} 
\usepackage{bbm} 
\usepackage[headings]{fullpage}
\usepackage{graphicx} 
\usepackage{tikz,pgfplots} 
\usepackage{hyperref} 
\usetikzlibrary{math}
\usepackage{setspace}
\usepackage{natbib}
\theoremstyle{plain}
\newtheorem{theorem}{Theorem}
\newtheorem{proposition}[theorem]{Proposition}
\newtheorem{corollary}[theorem]{Corollary}
\newtheorem{lemma}[theorem]{Lemma}

\theoremstyle{definition}
\newtheorem{definition}[theorem]{Definition}

\theoremstyle{remark}
\newtheorem*{remark}{Remark}

\def\E{\mathbb{E}}
\def\A{\mathcal{A}}
\def\B{\mathcal{B}}
\def\H{\mathcal{H}}
\def\K{\mathcal{K}}
\def\F{\mathcal{F}}
\def\St{\mathcal{C}}               
\def\es{\emptyset}

\def\mc{\mathcal}
\def\ms{\mathsf}

\def\one{\mathbbmss{1}}

\def\Q{\mathbb{Q}}
\def\P{\mathbb{P}}
\def\PP{\mathcal{P}}

\def\R{\mathbb{R}}
\def\Z{\mathbb{Z}}

\def\Rng{\ms{Rng}}					

\def\Vor{\ms{Vor}}					
\def\vCell{\ms{Cell}}
\def\conv{\ms{conv}}

\def\discr{\ms{discr}}				
\def\dist{\ms{dist}}				

\def\vOmega{\Omega}					
\def\vPhi{\Phi}						
\def\vPsi{\Psi}	
\def\vXi{\Xi}						
\def\vXir{\Xi^{\ominus r}}				

\def\valpha{\alpha}					
\def\vbeta{\beta}					
\def\vdelta{\delta}					

\def\vgamma{\gamma}					
\def\vla{\lambda}					
\def\vphi{\varphi}					
\def\vpsi{\psi}						
\def\vtau{\tau}						
\newcommand{\tauest}[2]{\widehat{\vtau}_{ #2}^{#1 }}		
\def\vzeta{\zeta}					

\def\vA{A}		
\def\vB{B}		
\def\vC{C}      
\def\vF{F}		
\def\vK{K}		
\def\vG{G}		
\def\vH{H}		
\def\vI{I}		
\def\vJ{J}		
\def\vK{K}		
\def\vL{L}		
\def\vN{N}		
\def\vQ{Q}		
\def\vR{R}		
\def\vS{\vH^{}_{\ms{O}}}		
\def\vT{T}		
\def\vTT{\mathbf{T}} 
\def\vU{U}		
\def\vV{V}		
\def\vW{W}		
\def\vX{X}
\def\vZ{Z}		

\def\va{a}	
\def\vd{d}		
\def\vf{f}		

\def\vk{k}
\def\vl{l}		
\def\vm{m}		
\def\vn{n}		
\def\vp{p}		
\def\vq{q}		
\def\vr{r}
\def\vt{t}		

\def\vv{e_d}		
\def\vx{x}		
\def\vy{y}		
\def\vz{z}		

\def\ceff{\sigma_{\mathrm{eff}}}					 
\def\sF{\sigma_{\mathcal{F}}}

\def\wt{\widetilde}
\def\rc{r_{\mathsf{c}}}								 
\def\Cinf{\mc{C}_\infty}							 
\newcommand{\vda}[1]{^{(#1)}}

\newcommand{\vWp}[1]{\vW^{}}
\newcommand{\vWpp}[2]{\vW^{#1}_{#2}}
\newcommand{\vRpp}[2]{\vR^{#1}_{#2}}
\newcommand{\vRminpp}[2]{\widetilde{\vR}^{#1}_{#2}}
\newcommand{\vWpI}{\vH^{}_{\ms{I}}}
\newcommand{\vWpIp}[1]{\vH^{}_{\ms{I},#1}}
\newcommand{\vWpIpp}[2]{\vH^{}_{\ms{I},#2}}
\newcommand{\vWpO}[1]{\vH^{}_{\ms{O}}}

\newcommand{\rmin}[2]{\vr_{\ms{min} #1{} #2}}
\newcommand{\rminest}[4]{{\widehat{\vr}}_{\ms{min}, #1{} #2 #3}^{#4}}
\newcommand{\rminestedge}[4]{{\widehat{\vr}}_{\ms{min}, #1{} #2 #3}^{\valpha}}
\newcommand{\rmax}{\vr_{\ms{max}}}
\newcommand{\estrmax}[1]{\widehat{\vr}_{\ms{max}, #1}}

\begin{document}
\title[Geodesic tortuosity and constrictivity]{Estimation of geodesic tortuosity and constrictivity in stationary random closed sets}
	
\author[Neumann, Hirsch, Stan\v{e}k, Bene\v{s}, Schmidt]{Matthias Neumann$^1$, Christian Hirsch$^2$, Jakub Stan\v{e}k$^3$, Viktor Bene\v{s}$^4$, Volker Schmidt$^1$}

\address{$^1$Institute of Stochastics, Ulm University}
\address{$^2$Department of Mathematics, Ludwig-Maximilians-Universit\"at M\"unchen}
\address{$^3$Department of Mathematics Education, Charles University Prague}
\address{$^4$Department of Probability and Mathematical Statistics, Charles University Prague}
	

\thanks{The work of MN, JS, VB, VS was partially funded by the German Science Foundation (DFG, project number SCHM 997/23-1) and the Czech Science Foundation (GACR, project number 17-00393J). The work of CH was funded by LMU Munich's Institutional Strategy LMUexcellent within the framework of the German Excellence Initiative.}

\begin{abstract}
We investigate the problem of estimating geodesic tortuosity and constrictivity as two structural characteristics of stationary random closed sets. They are of central importance for the analysis of effective transport properties in porous or composite materials. Loosely speaking, geodesic tortuosity measures the windedness of paths whereas the notion of constrictivity captures the appearance of bottlenecks resulting from narrow passages within a given materials phase. We first provide mathematically precise definitions of these quantities and introduce appropriate estimators. Then, we show strong consistency of these estimators for unboundedly growing sampling windows. In order to apply our estimators to real datasets, the extent of edge effects needs to be controlled. This is illustrated using a model for a multi-phase material that is incorporated in solid oxid fuel cells.
\end{abstract}

\maketitle

\section{Introduction}
\label{IntroSec}
The interplay between geometric microstructure characteristics and physical properties of materials motivates a major stream of current research activities. A better understanding of this relationship can help to increase the overall performance of functional materials as diverse as such used in solid oxide fuel cells, batteries and solar cells. This paper provides tools for mathematical modeling and statistical analysis of two crucial quantities: \emph{geodesic tortuosity} and \emph{constrictivity}.

In materials science, diverse definitions of tortuosity are used for the characterization of microstructures \citep{clennell.1997}, where tortuosity is sometimes defined as an effective transport property. The notion of geodesic tortuosity measures the lengths of shortest transportation paths with respect to the materials thickness and is particularly well-suited for algorithmic estimations from 3D image data~\citep{peyrega.2013}. Although this characteristic merely depends on the geometry of the underlying microstructure, it essentially influences effective transport properties of the material. Indeed, if the microstructure is set up such that the transportation paths between two interfaces are long and highly winded, then this serves as a first indication for poor quality of effective transport properties.

Despite the importance of path lengths, conductivity of materials is also influenced by more fine-grained characteristics of the material along the trajectories. For instance, transport processes are substantially obstructed by the presence of frequent narrow bottlenecks of transportation paths. To capture this effect, constrictivity quantifies the appearance of bottlenecks resulting from narrow passages. In the setting of tubes with periodically appearing bottlenecks, constrictivity is intimately related to the concept of effective diffusivity~\citep{petersen.1958}.
 More recently, the definition of constrictivity was extended to complex microstructures  based on the continuous pore size distribution~\citep{holzer.2013b}. The latter characteristic is directly related to the notion of granulometry in mathematical morphology \citep{matheron.1975}.

Virtual materials testing -- that is, the combination of stochastic microstructure modeling, image analysis and numerical simulation -- made it possible to empirically derive a quantitative relationship between  volume fraction $\vp$, mean geodesic tortuosity $\vtau$, constrictivity $\vbeta$ and the ratio of effective conductivity $\ceff$ to intrinsic conductivity  \citep{stenzel.2016}. Validation with experimental data from literature showed that the relationship can be used to predict effective conductivity in real microstructures by geometric microstructure characteristics. 

In order to understand geometric properties of the microstructure in materials via statistical analysis and simulation, stochastic geometry has emerged as a powerful, scalable and versatile framework, see e.g. \citep{chiu.2013, kendall.2010, ohser.2009}. Despite their numerous applications in the materials science literature, the notions of tortuosity and constrictivity have not yet been analyzed from a mathematical point of view. In fact, even a rigorous definition in the framework of stationary random closed sets is not available. This raises the question whether the empirical estimates in the literature for both, mean geodesic tortuosity \citep{gommes.2009, peyrega.2013, soille.2003, stenzel.2016} and constrictivity \citep{holzer.2013b, stenzel.2016} have any statistical underpinning, or whether the findings are purely heuristical.

In the present paper, we address this issue by providing mathematically precise definitions of these quantities as well as consistent estimators. The most naive version of the estimators is difficult to implement in practice, as it requires to take into account paths reaching arbitrarily far outside of the considered sampling window. Therefore, we establish sufficient conditions under which the estimators remain consistent in the setting where only a small amount of plus-sampling of the sampling window is required. Additionally, we illustrate how the conditions can be verified in the case of mollified Poisson relative neighborhood graphs, which have recently been used to model microstructures in solid oxide fuel cells \citep{n.2016}. Realizations of this model are simulated and the estimators of geodesic tortuosity and constrictivity are computed in order to illustrate the theoretical results. 

The paper is organized as follows. In Section~\ref{DefSec}, the notions of geodesic tortuosity and constrictivity as well as their estimators are defined in the framework of random closed sets. Moreover, in this section, the main results are presented. The corresponding proofs are given in Sections~\ref{measureSec} and \ref{EstSec}. In Section~\ref{EdgeEffects}, the influence of edge effects on the estimation of geodesic tortuosity and constrictivity is analyzed for a certain class of random closed sets using percolation theory. Finally, numerical simulations illustrate the theoretical results in Section~\ref{Numerics}.

\section{Definitions and main results}
\label{DefSec}
The proofs of the results stated in this section are postponed to Section~\ref{EstSec}.

\subsection{Preliminaries}
We consider the $d$-dimensional Euclidean space, $d \ge 2$. For each $\vB \subset \R^{d}$ the interior, the closure and the boundary of $\vB$ are denoted by $\mathring{\vB}, \bar{\vB}$ and $\partial \vB,$ respectively. Let $\F$ and $\K$ denote the families of closed and compact sets in $\R^{d}$, respectively. Furthermore, let $\nu_{\vd}$ be the $\vd$-dimensional Lebesgue measure and $\H_{\vk}$ be the $\vk$-dimensional Hausdorff measure for $1\leq \vk \leq \vd$. By $$\sF=\sigma(\lbrace \vF \in \F: \vF \cap \vK \neq \emptyset \rbrace: \vK \in \K \rbrace)$$ we denote the Borel $\sigma$-algebra with respect to the Fell topology \citep{fell.1962} on $\F$. The open and closed balls with radius $r > 0$ centered at $\vx \in \R^d$ are denoted by $b(x,r)$ and $B(x,r)$, respectively. By $\vWpI = \lbrace \vx \in \R^{\vd}: \vx_{\vd} = 0 \rbrace$ and $\vWpO{} = \lbrace \vx \in \R^{\vd}: \vx_{\vd} = 1 \rbrace$ we denote the hyperplanes orthogonal to the $d$-th standard unit vector $\vv = (0, \ldots, 0, 1)$ at distance $0$ and $1$ to the origin. The \emph{set of all paths going from $\vx \in \R^d$ to $\vF_1 \in \F$} via the interior of $\vF_{0} \in \F$ is denoted by
$$ \PP_{\vF_{0}}(x, \vF_{1}) = \lbrace \vf:\,[0,1] \longrightarrow \mathring{\vF_{0}} \text{ Lipschitz}: \vf(0) = x, \vf(1) \in \vF_{1}\rbrace.$$
Note that the Lipschitz condition ensures that paths are rectifiable, see \citep[p.~251]{federer.1969}.
Finally, $$\mathcal{C}_{F_0}(F_1)=\lbrace \vx \in \R^d : \PP_{F_0}(x , F_1) \neq \emptyset \rbrace$$
denotes the set of all points connected to a closed set  $F_1 \in \mc{F}$ through $\mathring{F_0}$. In other words, $\mathcal{C}_{F_0}(F_1)$ describes the \emph{union of the connected components of $F_0$ intersecting $F_1$}. Throughout the paper, we consider a complete probability space $(\vOmega, \A, \P)$, which is not further specified.

\subsection{Definitions}
\subsubsection{Mean geodesic tortuosity}
\label{DefTauSec}
The notion of mean geodesic tortuosity for stationary random closed sets quantifies the windedness of directional paths through the considered sets. Among the various concepts of tortuosity in materials science, the concept of mean geodesic tortuosity constitutes an important microstructure characteristic for the prediction of effective conductivity in multi-phase materials~\citep{stenzel.2016}. 

 In the following, we assume that the direction of transport is given by $e_d$. Tortuosity in other directions reduces to this setting after a suitable rotation of the underlying random closed set. Intuitively speaking, mean geodesic tortuosity of a stationary random closed set $\vXi$ in $\R^{d}$ is defined as the expectation of the length of the shortest path in $\mathring{\vXi}$ from the origin $o \in \mathbb{R}^{d}$ to the hyperplane $l\vWpO{}$ under the condition that at least one such path exists. For normalization, mean geodesic tortuosity is divided by $l$. An illustration of a shortest path is given in Figure~\ref{fig:tortuosity_sketch}, left. After rescaling $\Xi$ suitably by $\vl^{-1}$, i.e., considering $\vl^{-1} \vXi = \lbrace \vl^{-1} x: x \in \vXi \rbrace$, we may assume that $\vl = 1$ in the following.  
 
\begin{figure}[h]
	\begin{center}
		\includegraphics[width=0.48\textwidth]{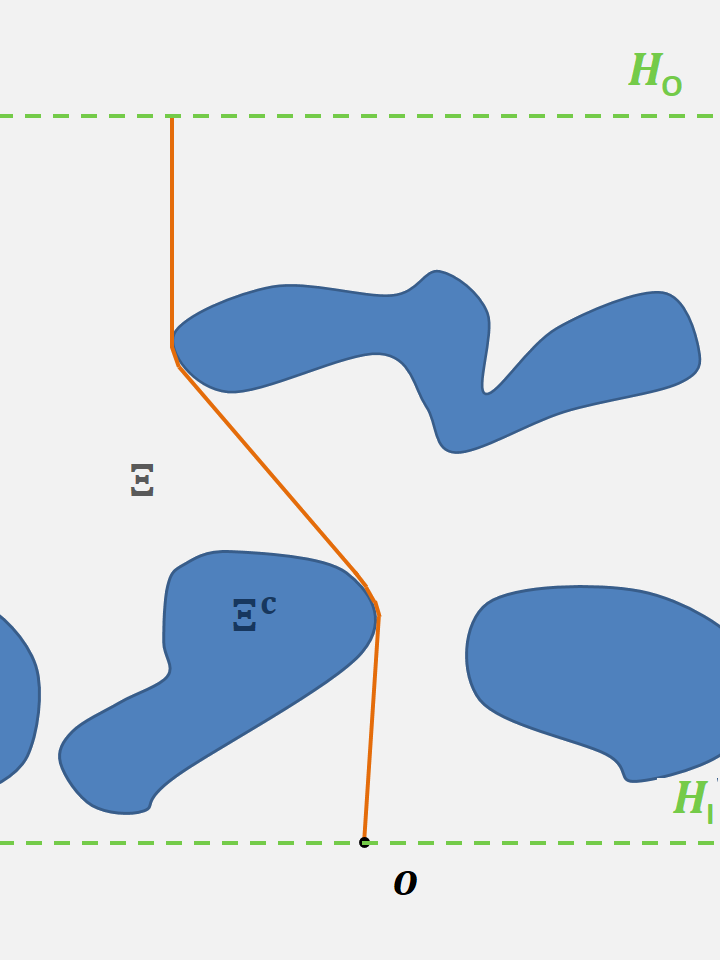}\quad
		\includegraphics[width=0.48\textwidth]{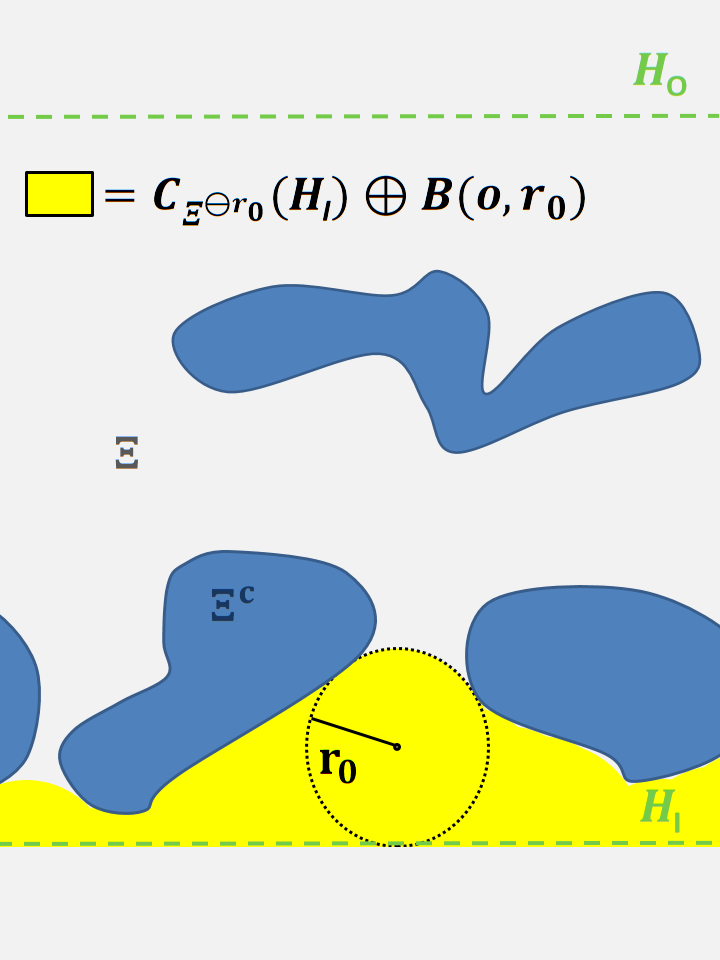}
	\end{center}
	
    \caption{Left: Visualization of the shortest path (orange) in $\Xi$ (gray) from the origin to $\vWpO{}$. To estimate the mean geodesic tortuosity  $\vtau$, the average length of all shortest paths going from $\vWpI$ to $\vWpO{}$ is considered. Right: Illustration of $\St_{\vXi^{\ominus \vr_0}}(\vWpI) \oplus B(o,\vr_0)$.
    }\label{fig:tortuosity_sketch}
\end{figure}

In order to define the mean geodesic tortuosity of a random closed set $\vXi$ in $\R^d$ rigorously, we need to be able to measure lengths of connection paths in $\vXi$ and determine whether different parts of $\vXi$ are contained in a common connected component. Recall that $\St_\Xi(F)$ describes the union of the connected components of $\vXi$ intersecting a set $F$. Furthermore, 
$$ \vgamma_{\vF}(x, \vT) = \inf_{f \in \PP_{\vF}(x , \vT)}  \H_1(\vf([0, 1]))$$
denotes the length of the shortest path contained in the interior of $\vF \in \F$ from $x$ to a target set $\vT \in \F$.

 Although we consider the paths in the interior of the random closed set for technical reasons, this definition is sensible for applications since transport through infinitely thin paths is not possible. In applications, it is sometimes assumed that paths are restricted in the half-space of non-negative last coordinate and many of our results carry over to this setting.

A random closed set is called $\R^{d-1}$-stationary if it is stationary with respect to shifts by vectors of the linear sub-space $\vWpI$.
\begin{definition}
    \label{def:def_tau}
    Let $\vXi$ be an $\R^{d-1}$-stationary random closed set. The \emph{mean geodesic tortuosity} of $\vXi$ is then defined by 
\begin{equation}\label{eq:def_tau}
\tau=\begin{cases}
    \, \E \left[ \vgamma_{\vXi}( o , \vWpO{}) \mid o \in \St_\vXi(\vWpO))\right]&
    \text{if } \P(o \in \St_\vXi(\vWpO))) > 0, \\
\infty & \text{otherwise.}
\end{cases}
\end{equation}
\end{definition}

Since, a priori, there are uncountably many paths, it is not clear that $\vgamma_{\vXi }( o, \vWpO{} )$ is a random variable and $\St_{\vXi}(\vWpO{} )$ is a measurable set. This is discussed in Proposition~\ref{prop:tau_welldef} and in the remark thereafter. Note that in the definition of $\tau$, it is sufficient to assume that $\vXi$ is $\R^{d-1}$-stationary. In case that $\vXi$ is $\R^{d-1}$-stationary, but not $\R^{d}$-stationary, $\tau$ is a local characteristic of $\vXi$. This can be relevant for the investigation of microstructures exhibiting a structural gradient. If $\vXi$ is additionally $\R^{d}$-stationary, $\tau$ is invariant under translations of the origin and thus $\tau$ is a global characteristic of $\vXi$.

\subsubsection{Constrictivity}
\label{DefBetaSec}
The notion of constrictivity of a stationary random closed set measures the strength bottleneck effects. This characteristic was introduced in materials science for tubes with periodically appearing bottlenecks in \citep{petersen.1958}, where constrictivity has been defined in dimension $\vd = 3$ as the ratio of the minimum and maximum area, through which transport goes. In \citep{petersen.1958} the minimum as well as the maximum area are circular areas with radii $\rmin{}{}$ and $\rmax$. Thus constrictivity is defined as $(\rmin{}{}/\rmax)^{2}$. Note that the concept of constrictivity can be transfered from simple geometries to complex microstructures \citep{holzer.2013b}, i.e.  $\rmin{}{}$ and $\rmax$ are defined for complex microstructures based on the concept of the continuous pore size distribution \citep{muench.2008}, which is directly related to the granulometry function of mathematical morphology \citep{matheron.1975}. As in Section \ref{DefTauSec} we assume that the transport direction is $e_d$. Constrictivity with respect to other transport directions can be reduced to this setting by a suitable rotation of the random closed set. When generalizing the concept of constrictivity to an arbitrary dimension, the definition changes to $(\rmin{}{}/\rmax)^{d-1}.$ The exponent $d-1$ appears in the definition of constrictivity as transport in $\R^{d}$ towards a predefined direction goes through $(d-1)-$dimensional cross-sections.

In the following $\vB^{\ominus \vr}$ denotes the \emph{erosion} $\vB \ominus B(o,r)$ of a set $\vB \subset \R^{d}$ by $B(o,r)$ for each $\vr > 0$. The Minkowski addition of two sets $\vB_{1}, \vB_{2} \subset \R^{\vd}$ is denoted by $\vB_{1} \oplus \vB_{2}$. To quantify bottleneck effects in a closed set $\vF \in \F$, we consider the set
$\St_{\vF^{\ominus \vr} }(\vWpI)$ consisting of all $\vx \in \vF$ such that $\vx$ can be reached by a path in the interior of $\vF^{\ominus r}$ starting at $\vWpI$. Then, $\St_{\vF^{\ominus \vr} }(\vWpI) \oplus B(o,r)$ is the subset of $\vF$ that can be filled by spheres starting from $\vWpI$, i.e. the centers of spheres are in $\vWpI$, and rolling freely in $\vF$. For an illustration, see Figure~\ref{fig:tortuosity_sketch}, right.

Next, for an $\R^{d-1}$-stationary random closed set $\vXi,$ by
  $$ r_{\mathsf{min}, \vl} = \sup \lbrace \vr \geq 0: \tfrac1l \, \E[\nu_{\vd} ((\St_{\vXi^{\ominus \vr}}(\vWpI) \oplus B(o,\vr)) \cap ([0,1]^{\vd-1} \times [0,\vl]))] \geq \tfrac12  \, \E[\nu_d(\vXi \cap [0, 1]^d)]\rbrace$$
  we denote the largest radius $r$ such that in expectation at least half  of $\vXi\cap ([0, 1]^{\vd-1} \times [0, l])$ can be filled by an intrusion of balls with radius $\vr$, i.e., by an intrusion from $\vWpI$  to $\vWpO)$. Note that the intrusion in transport direction determining $r_{\mathsf{min}, \vl},$ is strongly influenced by bottlenecks in $\vXi$.

  In order to obtain a quantity invariant under rescaling of $\vXi$, the radius $r_{\mathsf{min}, \vl}$ must be related to the overall thickness of $\vXi$. More precisely, writing 
  $\Psi_r(\Xi) = \Xi^{\ominus r} \oplus B(o, \vr)$
  for the \emph{opening of $\Xi$}  let
 $$\rmax = \sup\lbrace \vr \geq 0 : \, \E[\nu_d( \Psi_r(\vXi) \cap [0, 1]^d)] \geq \tfrac12 \E[\nu_d( \vXi\cap [0, 1]^d)] \rbrace$$
 denote the largest radius $\vr$ such that in expectation at least half of $\vXi$ can be covered by balls of radius $\vr$ entirely contained within $\vXi$.
 
 \begin{definition}
     \label{def:constr}
     Let $\vXi$ be an $\R^{d-1}$-stationary random closed set. The \emph{constrictivity} $\vbeta_{\vl}$ of $\vXi$ is then defined by
     $ \vbeta_{ \vl} = (r_{\mathsf{min}, \vl}/\rmax)^{d-1}.$
  \end{definition}
  Since the constrictivity $\vbeta_{\vl}$ of $\vXi$ is identical to the constrictivity $\vbeta = \vbeta_1$ of the scaled set $\vl^{-1}\vXi$, we only consider the case $\vl = 1$ from now on. Conceptually $\vbeta$ is a measure for the strength of bottleneck effects. Typically, there are many narrow constrictions in $\vXi$ if $\beta$ is close to 0, whereas if $\beta = 1$, then there are no constrictions at all.  If $\mathring{\vXi}$ is almost surely connected, then $\rmin{}{} \leq \rmax$ and $0 \leq \vbeta \leq 1$. Otherwise, it can happen $\rmin{}{} = -\infty$, as can be seen in the case where $\vXi$ is almost surely a union of disjoint sets of diameter smaller than 1.

\subsection{Construction of consistent estimators}
\label{estSec}
In the following, let $\vXi$ be an $\R^{d-1}$-stationary and ergodic random closed set. That is, $\vXi$ is stationary and ergodic with respect to the group of translations $\vTT=  \{\vTT_{x}\}_{ x \in \vWpI} $, where $\vTT_x:\Omega \longrightarrow \Omega$ is a $\P$-invariant mapping such that $\vXi(\vTT_{x} \omega) = \vXi(\omega) - x$. Now, we construct estimators for mean geodesic tortuosity and constrictivity of a random closed set observed in a bounded sampling window $\vWpp{}{\vN} = [-N/2, N/2]^{d-1} \times [0,1]$ for some integer $N > 0$.

\subsubsection{Mean geodesic tortuosity}

To estimate the mean geodesic tortuosity $\tau$, we consider paths starting in the window
$\vWpIpp{}{\vN} = \vWpI{} \cap \vWpp{}{\vN}.$
Then, we  define the estimator $\tauest{}{\vN}$ for mean geodesic tortuosity as
\begin{align*}
    \tauest{}{\vN} = \frac{1}{\H_{\vd-1}(\St_{\vXi}( \vWpO{})\cap \vWpIpp{}{\vN})} \int_{\St_\vXi(\vWpO))\cap \vWpIpp{}{\vN}} \vgamma_{\vXi}( \vx , \vWpO{\vl}) \, \mathcal{H}_{d-1}(\mathrm{d}\vx).
\end{align*}

We prove strong consistency of $\tauest{}{\vN}$ as $N \rightarrow \infty$ for $\R^{d-1}$-stationary and ergodic random closed sets under some moment condition of the shortest path-lengths.

\begin{theorem}\label{thm:consistency_tau1}
    Let $\E [\vgamma_{\vXi}(o, \vWpO{\vl}) \one_{o \in \St_\vXi( \vWpO{\vl})}] < \infty$. Then, $\tauest{}{\vN}$ defines a strongly consistent estimator of $\vtau$. That is, $\tauest{}{\vN}$ converges almost surely to $\vtau$ as $N \to \infty$.
\end{theorem}

Using the estimator $\tauest{}{\vN}$ requires information about the length of all shortest paths from  $\vWpIpp{\vl}{\vN}$ to $\vWpO{\vl}$ through $\vXi$. In practice, $\vXi$ is observed in a bounded sampling window, which does not necessarily contain all shortest paths that are required to compute $\tauest{}{\vN}$. Thus we consider a further estimator for $\tau$. For estimating mean geodesic tortuosity based on a bounded sampling window, we observe paths from $\vWpIpp{\vl}{\vN}$ to $\vWpO{\vl}$ going through the dilated window
$$\vWpp{\valpha}{\vN}  = \vWpp{}{\vN} \oplus ([-\vN^{\valpha}, \vN^{\valpha}]^{\vd - 1} \times [-\vN^{\valpha},0])$$ 
for some $\valpha > 0$. Concerning all paths going through a dilated sampling window reduces the edge effects. Indeed, we also take paths from $\vWpIpp{\vl}{\vN}$ to $\vWpO{}$ into account leaving $\vWpp{}{\vN}$, which tend to be longer than paths completely contained in $\vWpp{}{\vN}$.

Hence, we consider the estimator
\begin{equation}\label{eq:tau_est_alpha}
    \tauest{\valpha}{\vN} = \frac{1}{\H_{\vd-1}(\St_{\vXi\cap \vWpp{\valpha}{\vN}}( \vWpO{})\cap \vWpIpp{}{\vN})} \int_{\St_{\vXi \cap \vWpp{\valpha}{\vN}}( \vWpO{})\cap \vWpIpp{}{\vN}} \vgamma_{\vXi\cap \vWpp{\valpha}{\vN}}(x, \vWpO{}) \, \mathcal{H}_{d-1}(\mathrm{d}\vx).
\end{equation}

In contrast to $\tauest{}{\vN},$ estimation of $\tau$ by means of $\tauest{\valpha}{\vN}$ takes those shortest paths into account, which are completely contained in the extended sampling window $\vWpp{\valpha}{\vN}$. Under some further assumptions regarding the shortest paths in $\vXi$ we obtain a result on consistency of the estimator $\tauest{\valpha}{\vN}$. Therefore, we define 
$\vRpp{\valpha}{\vN} =\{\vgamma_{\vXi}(x, \vWpO{}) =\vgamma_{\vXi\cap \vWpp{\valpha}{\vN}}(x, \vWpO{}) \text{ for all $x \in \vWpIp{\vN} \cap \Q^d$}\}  $
as the event that all shortest paths going from $\vWpIp{\vN}$ to $\vWpO{l}$ are completely contained in  $\vWpp{\valpha}{\vN}$, where $\Q$ denotes the set of rational numbers.

\begin{corollary}\label{thm:consistency_tau2}
        Let $\E [\vgamma_{\vXi}(o, \vWpO{\vl}) \one_{o \in \St_\vXi( \vWpO{\vl})}] < \infty$. Then, the following statements are true:
        \begin{enumerate}
            \item If there exists an almost surely finite random variable $\vN_0 \ge 1$ such that the event $\vRpp{\valpha}{\vN}$ occurs for all $\vN \ge \vN_0$, then the estimator $\tauest{\valpha}{\vN}$ is strongly consistent as $N \rightarrow \infty$.
            \item If $\lim_{N \rightarrow \infty}\P(\vRpp{\valpha}{\vN})=1$, then the estimator  $\tauest{\valpha}{\vN}$ is weakly consistent as $N \rightarrow \infty$.
        \end{enumerate}
\end{corollary}

\subsubsection{Constrictivity}
In order to estimate the constrictivity $\beta_{ }$, we introduce estimators for $\rmax$ and $\rmin_{}$. In particular, for $\rmax$, we define the estimator

\begin{align*}
    \estrmax{N}^{} = \sup \left\lbrace \vr \geq 0: 2\nu_{d}\left(\Psi_r(\vXi) \cap \vWpp{}{\vN}\right) \geq  \nu_{d}\left( \vXi \cap\vWpp{}{\vN}\right) \right\rbrace.
\end{align*}

\begin{theorem}\label{thm:consistency_rmax}
    If there exists at most one $r_0 > 0$ with
\begin{equation}\label{eq:assumption_consistency_rmax}
    2 \, \E[\nu_{d}( \Psi_{r_0}(\vXi) \cap \vWpp{}{1})] = \E[\nu_d(\vXi \cap \vWpp{}{1})],
\end{equation} 
    then the estimator $\estrmax{\vN}$ is strongly consistent as $N \rightarrow \infty$.
    \end{theorem}

To estimate $\rmin{}{}$, we put 
$ \vZ_{\vXi,\vr} = \St_{\vXi^{\ominus \vr}}(\vWpI) \oplus B(o,\vr)$
and define the estimator
\begin{align*}
\rminest{}{}{\vN}{} = \sup \lbrace \vr \geq 0: 2\nu_{\vd}(\vZ_{\vXi,\vr}  \cap\vWpp{}{\vN}) \geq \nu_{\vd}(\vXi \cap \vWpp{}{\vN}) \rbrace,
\end{align*}

\begin{theorem}\label{thm:consistency_rmin}
	If there exists at most one $r_0 > 0$ with
	\begin{equation}\label{eq:assumption_consistency_rmin}
        2 \, \E[\nu_{d}( \vZ_{\vXi,\vr_0} \cap \vWpp{}{1})]= \E[\nu_d( \vXi \cap \vWpp{}{1})],
	\end{equation} 
    then the estimator $\rminest{}{}{\vN}{}$ is strongly consistent as $N \rightarrow \infty$.
\end{theorem}

In practice, it might be difficult to verify Conditions (\ref{eq:assumption_consistency_rmax}) and (\ref{eq:assumption_consistency_rmin}) for a given random closed set. Thus, we present a further more accessible sufficient condition.

\begin{corollary}\label{cor:consistency_rmin_weaker_assumption}
	Assume that
	\begin{equation}\label{eq:weaker_assumption_consistency_rmin}
        \P( o \in \vXi^{\ominus \vr_1} \setminus \vPsi_{r_2}(\vXi)) > 0
	\end{equation} 
	\begin{enumerate}
		\item for all $0 < \vr_{1} < \rmax < \vr_{2}.$
		Then, Condition \eqref{eq:assumption_consistency_rmax} is satisfied.
		\item for all $0 < \vr_{1} < \rmin{}{} < \vr_{2}.$
		Then, Condition \eqref{eq:assumption_consistency_rmin} is satisfied.
	\end{enumerate}
	
\end{corollary}

For the usage of $\rminest{}{}{\vN}{}$ the complete information about $\vZ_{\vXi,\vr}$ is required for each $r > 0$. In applications, $\vXi$ is observed in a bounded sampling window. Then $\vZ_{\vXi,\vr}$ has to be determined based on the observation of $\vXi$. This is taken into account by the estimator \begin{equation}\label{eq:rmin_est_alpha}
\rminestedge{}{}{\vN}{} = \sup \lbrace \vr \geq 0: 2\nu_{\vd}(\vZ_{\vXi \cap \vWpp{\valpha}{\vN}, \vr}  \cap\vWpp{}{\vN}) \geq \nu_{\vd}(\vXi \cap \vWpp{}{\vN}) \rbrace, 
\end{equation} 
where $\vZ_{\vXi \cap \vWpp{\valpha}{\vN},\vr}$ is used instead of $\vZ_{\vXi,\vr}$ to estimate $\rminest{}{}{\vN}{}$, since $\vZ_{\vXi \cap \vWpp{\valpha}{\vN},\vr}$ can be determined based on a observation of $\vXi$ in $\vWpp{\valpha}{\vN}$. Under further assumptions on the connected components of $\vXi^{\ominus \vr}$ we obtain a consistency result of the estimator $\rminestedge{}{}{\vN}{2}$. Therefore, we define for each $r > 0$ the event 
$\vRminpp{\valpha}{\vN,r} = \{\St_{\vXi^{\ominus \vr}}(\vWpIp{\vN}) = \St_{\vXi^{\ominus \vr}\cap \vWpp{\valpha}{\vN}}(\vWpIp{\vN})\}$ 
that each $x \in \vXi^{\ominus \vr}$ can either be connected to $\vWpIp{\vN}$ by a path within $\vXi^{ \ominus \vr} \cap \vWpp{\valpha}{\vN}$ or is contained in a connected component, which does not intersect $\vWpIpp{}{\vN}.$

\begin{corollary}
    \label{constCor}
    Let Condition (\ref{eq:assumption_consistency_rmin}) be fulfilled. 
                    \begin{enumerate}
                        \item If there exists an almost surely finite random variable $\vN_0 \ge 1$ such that the event $\vRminpp{\valpha}{\vN,r}$ occurs for all $\vN \ge \vN_0$, then the estimator $\rminestedge{}{}{\vN}{}$ is strongly consistent as $N \rightarrow \infty$.
                        \item If $\lim_{N \rightarrow \infty}\P(\vRminpp{\valpha}{\vN})=1$, then the estimator $\rminestedge{}{}{\vN}{}$ is weakly consistent as $N \rightarrow \infty$.
                    \end{enumerate}
\end{corollary}

\section{Measurability}
\label{measureSec}
\subsection{Geodesic tortuosity}
\label{measureTauSec}
 In the following we show the well-definedness of geodesic tortuosity $\tau$ as formalized in the following result.

\begin{proposition}\label{prop:tau_welldef}
    The functions $F \mapsto \one\{o \in \St_{\vF}(\vWpO{})  \}$ and $F \mapsto \vgamma_{\vF}(o, \vWpO{})$ are $(\sigma_\F, \bar{\B})$-measurable. In particular, the conditional expectation $\vtau$ given in \eqref{eq:def_tau} is well defined.
\end{proposition}
     \begin{remark}
    	In Section \ref{measureConstrSec}, it is shown that if $\Xi$ is a random closed set, then so is $\R^d \setminus \St_{\vXi}(\vWpI{})$. Note that this result can be analogously obtained for $\R^{d} \setminus \St_\Xi(\vWpO{})$. 
    \end{remark}
    
 That is, we show that $\vgamma_{\vXi}(o,\vS)$ is a random variable with values in $\R \cup \lbrace \infty \rbrace$ and that 
 $\{o \in \St_{\vXi}(\vS)\} \in \A$ for each random closed set $\vXi$. For this purpose, continuous paths from $o$ to $\vS$ through $\vXi_1$ are approximated by line segments, which is a common approach in the literature \cite{davis.2017}. In the following, we denote the line segment between $\vx$ and $\vy$ by
 $$[\vx,\vy]=\lbrace \vz=\vx+\va(\vy-\vx): \va \in [0,1]\rbrace,$$
 for all $\vx,\vy$ in $\R^{\vd}$.    
    
Since our proof of measurability of geodesic tortuosity relies on an approximation by line segments, we first show that the event that a line segment is contained in a random closed set is measurable.
\begin{lemma}\label{lem:straightline}
    Let $\vx, \vy \in \R^{\vd}.$ Define the mapping $\vpsi_{\vx, \vy}:\F \to \R \cup \lbrace -\infty \rbrace$ by
    $$ \vF \mapsto \begin{cases}
    |\vx-\vy|  & \text{if } [\vx,\vy] \subset \vF, \\
    -\infty  & \text{otherwise.}
    \end{cases}$$
    Then, $\vpsi_{x,y}$ is $(\sF, \bar{\B})$-measurable, where
    $\bar{\B}=\lbrace \vA \cup \vA': \vA \in \B(\R), \vA' \subset \lbrace -\infty, \infty \rbrace \rbrace \cup \B(\R).$
\end{lemma}
\begin{proof}
    Since $\psi_{x,y}$ is a piecewise constant function, it suffices to show that $\{[\vx,\vy] \subset \vF\}$ is measurable with respect to $\vF$. 
    Considering the rational approximation $$[\vx,\vy]_{\Q} = \lbrace \vz=\vx + \va(y-x): \va \in \Q \cap [0,1] \rbrace,$$ we have
    $$ \lbrace \vF \in \F: [\vx,\vy] \subset \vF \rbrace = \bigcap_{\vz \in [\vx,\vy]_{\Q}} \lbrace \vF \in \F: \lbrace \vz \rbrace \cap \vF \neq \emptyset \rbrace \in \sF,$$
    which completes the proof.
\end{proof}

\begin{proof}[Proof of Proposition~\ref{prop:tau_welldef}]
    Since $\lbrace o \in \St_\Xi(\vWpO{})\rbrace = \lbrace \vgamma_{\vXi }(o,\vS) > -\infty \rbrace \in \A,$
    it suffices to prove measurability of $\vgamma_{\vF}(o,\vT)$.    By the previous remark on rectifiability, the length of a path $\vf \in \PP_{\vF}(o, \vT)$ can be approximated by the sum of the length of line segments connecting points on the curve $\vf([0,1])$. Moreover, since $\vf \in \PP_{\vF}(o, \vT)$ contains only paths contained in the interior of $\vF$, the line segments can be assumed to be contained in $\vF$. That is, 
    \begin{align*}
    \H_{1}(f([0,1])) &= \sup_{k \geq 1} \,  \sup \left\lbrace   \sum_{i \le \vk } |\vf(\vt_{i-1}) - \vf(\vt_{i})|  : 0 = \vt_{0} < \vt_{1} < \cdots < \vt_{k} = 1, \right. \\ & \qquad  [f(\vt_{i-1}), f(\vt_{i})] \subset F \text{ for each } i \leq k    \Bigg\rbrace  \\
    &= \sup_{k \geq 1} \, \sup_{0 = \vt_{0} < \vt_{1} < \cdots < \vt_{k} = 1} \sum_{i \le \vk - 1} \vpsi_{\vf(\vt_{i-1}), \vf(\vt_{i})}(F) \\
    &= \sup_{k \geq 1} \, \sup_{\substack{0 = \vt_{0} < \vt_{1} < \cdots < \vt_{k} = 1, \\ \vt_{1}, \ldots, \vt_{k-1} \in \Q}} \sum_{i\le k} \vpsi_{\vf(\vt_{i-1}), \vf(\vt_{i})}(F).
    \end{align*}
    Thus, defining $\PP'(y)$ to be the family of all piecewise affine linear functions $\vf:[0, 1] \to \mathring{F}$ that have coefficients in $\Q$ and satisfy $\vf(0)=o$, $\vf(1)=y$, we can approximate paths by line segments to obtain that
    \begin{align*}
        \vgamma_{\vF}(o,\vT) &= \inf_{\vy \in \vT} \, \inf_{\vf \in \PP_{\vF}(o, \lbrace \vy \rbrace)} \H_{1}(f([0,1])) = \inf_{\vy \in \vT \cap \Q^{\vd}} \, \inf_{\vf \in \PP'(\vy)} \H_{1}(f([0,1])) \\
        &= \inf_{\vy \in \vT \cap \Q^{\vd}} \, \inf_{\vf \in \PP'(\vy)} \sup_{k \geq 1} \, \sup_{\substack{0 = \vt_{0} < \vt_{1} < \cdots < \vt_{k} = 1 \\ \vt_{1}, \ldots, \vt_{k-1} \in \Q}} \sum_{i=0}^{\vk - 1} \vpsi_{\vf(\vt_{i}), \vf(\vt_{i + 1})}(F)
    \end{align*}
    Hence, $\vgamma_{\vF}(o,\vT)$ can be represented via nested infima and suprema of countably many functions that are measurable by Lemma~\ref{lem:straightline}.
         Thus $\vgamma_{\vF}(o,\vT)$ is measurable as an infimum of countably many measurable functions. Since $\vXi$ is $(\A, \sF)$-measurable, $\vgamma_{\vXi}(o,\vS)$ is $(\A, \bar{\B})$-measurable. Moreover this is which leads to the claim.
\end{proof}

\subsection{Constrictivity}
\label{measureConstrSec}
In the following it is shown that the constrictivity $\vbeta = (\rmax/\rmin{}{})^{d-1}$ is well defined for random closed sets $\vXi$. The well-definedness of $\rmax$ can be deduced directly from basic properties of random closed sets.

 \begin{lemma}\label{lem:compositionRACS}
    Let $\vF \in \F, \vK \in \K$ be arbitrary. Then, $\overline{\vXi^{c}}, \vXi \cap \vF, \vXi \cup \vF,  \vXi \oplus \vK$ and $\vXi \ominus \vK$ are random closed sets.
 \end{lemma}
 \begin{proof}
     The assertions follow from~\citep[Chapter 1, Thm.~2.25]{molchanov.2005} since $\vF \oplus \vK$ and $\vF \ominus \vK$ are closed for all $\vF \in \F, \vK \in \K$.
 \end{proof}

 Showing that $\rmin{}{}$ is well defined involves further arguments and is summarized in the following result.
\begin{proposition}
    \label{prop:connMeas}
    The function $\vF \mapsto \R^{d} \setminus  \St_{\vXi^{\ominus r}}(\vWpI{})$ is $(\sigma_{\F}, \sigma_{\F})$-measurable.
\end{proposition}

Since
$$\R^{d} \setminus  \St_{\vF^{\ominus r}}(\vWpI{}) = (\vF^{\ominus r})^{c} \cup (\vF^{\ominus r} \setminus \St_{\vF^{\ominus r}}(\vWpI{}))= (\vF^{\ominus r})^{c} \cup \lbrace \vx \in \vF^{\ominus r}: \PP_{\vF^{\ominus r}}(\vx, \vWpI) = \emptyset \rbrace,$$
by Lemma \ref*{lem:compositionRACS}, it is sufficient to show that 
$\widetilde{\vZ}_{\vF} = \lbrace \vx \in \vF: \PP_{\vF}(\vx, \vWpI) = \emptyset \rbrace$
is measurable in $\vF$.

The idea of the proof is to represent $\widetilde{\vZ}_{\vF}$ as sub-level set of a measurable lower semicontinuous function and apply~\citep[Chapter 5, Proposition 3.6]{molchanov.2005}. For this purpose, we define the function
$\vzeta: \R^{\vd} \times \F \longrightarrow \R$ by
$$
(\vx, \vF) \mapsto 
\begin{cases}
    \vgamma_{\vF}( \vx , \vWpI{}) & \text{if } \vgamma_{\vF}(\vx,\vWpI{} ) < \infty, \\
\hfil -1 & \text{otherwise.}
\end{cases}$$
To simplify notation, we write $\vzeta(\vx)=\vzeta(\vx, \vF)$ and as in Proposition \ref{prop:tau_welldef} it can be shown that $\vzeta(\vx)$ is measurable in $\vF$ for each $\vx \in \R^{d}.$ Then, $\widetilde{\vZ}_{\vF}$ can be expressed in terms of $\vzeta$, in particular
\begin{equation}\label{eq:levelset_reprsentation}
\widetilde{\vZ}_{\vF} = \lbrace \vx \in \R^{\vd}: \vzeta(\vx) \leq -1 \rbrace.
\end{equation}
The following two lemmas are used to show that $\widetilde{\vZ}_{\vF}$ is measurable.

\begin{lemma}\label{lem:semicontinuity}
The function  $\vzeta$ is lower semicontinuous. That is,
    $$\liminf_{\vy \rightarrow \vx}  \vzeta(\vy, \vF) \geq \vzeta(\vx, \vF).$$
\end{lemma}
\begin{lemma}\label{lem:jointmeasurability}
    The function $\vzeta$ is $(\B(\R^{\vd}) \otimes \sigma_\F, \B(\R))$-measurable.
\end{lemma}
Before proving the auxiliary results, we explain how to complete the proof of Proposition~\ref{prop:connMeas}.
\begin{proof}[Proof of Proposition~\ref{prop:connMeas}]
    By identity~\eqref{eq:levelset_reprsentation}, $\widetilde{\vZ}_{\vF}$ is the sub-level set of the function $\vzeta$, which is lower semicontinuous and measurable by Lemmas \ref{lem:semicontinuity} and \ref{lem:jointmeasurability}.  Therefore,~\citep[Chapter 5, Proposition 3.6]{molchanov.2005} implies that
        $\lbrace (\vx,\vt) \in \R^d \times \R: \vzeta(\vx, \vF) \leq \vt \rbrace$
        and therefore
        $\lbrace \vx \in \R^d: \vzeta(\vx, \vF) \leq -1 \rbrace$ are measurable in $\vF$.
\end{proof}

It remains to establish lower semicontinuity and measurability of $\zeta$.

\begin{proof}[Proof of Lemma~\ref{lem:semicontinuity}]
    It suffices to consider the case, where $\vzeta(\vx, \vF) > -1$. Let now $\vx \in \mathring{\vF}$ with  $\PP_{\vF}(\vx, \vWpI{}) \ne \es.$ Furthermore, let $\vx_{1}, \vx_{2},\ldots$ be a sequence in $\R^{\vd}$ with $\lim_{\vn \rightarrow \infty} \vx_{\vn} = \vx.$ Since $\vx \in \mathring{\vF}$ there exists an $\vr_{0}>0$ such that there is an $\vn_{0}$ with $x_{\vn} \in b(\vx,\vr_{0})$ for each $\vn \geq \vn_{0}$ and $b(\vx,\vr_{0})\subset \mathring{\vF}$. Then it holds that
    $|\vzeta(\vx_{\vn}, \vF) - \vzeta(\vx, \vF)| \leq |\vx_{\vn} -\vx |$ for each $\vn_{0} \geq \vn$. Thus $\vzeta$ is continuous at $\vx$.
\end{proof}

\begin{proof}[Proof of Lemma~\ref{lem:jointmeasurability}]
    The proof is strongly leaned on the method used for the proof of~\citep[Theorems 2 and 3]{gowrisankaran.1972}. For each $\va \in \R$, define the strict sub-level set
    $$\vJ_{a} = \lbrace (\vx, \vF) \in \R^{\vd} \times \F: \vzeta(\vx, \vF) < \va \rbrace.$$
    Then, $\vJ_{\va}=\emptyset \in \B(\R^{\vd}) \otimes \sigma_\F$ if $\va \le -1$. Furthermore, for every $\va \in (-1, 0]$ we have 
    \begin{align*}
    \vJ_{\va} &= \lbrace (\vx, \vF) \in \R^d \times \F:\, \vzeta(\vx, \vF) = -1 \rbrace  \\
    &= \lbrace (\vx, \vF) \in \R^d \times \F:\, \vzeta(\vx, \vF) > -1 \rbrace^{c} \\
    &= \left( \bigcup_{\vn \geq 1} \bigcap_{\vk > \vn} \bigcup_{\vq \in \Q^{\vd}} \left(b(\vq, 1/\vk) \times \lbrace \vF \in \F:\, \vzeta(\vq, \vF) > -1 \rbrace \right) \right)^{c} \in \B(\R^{\vd}) \otimes \sigma_\F,
    \end{align*}
        where the third equality follows from the openness of $\lbrace x \in \R^{d}: \vzeta(\vx, \vF) > -1 \rbrace$. Finally, for every $\vF \in \F$ the mapping $\vzeta(\cdot, \vF): \R^{d} \longrightarrow \R$ is continuous at each $\vx \in \R^{d}$ with $\vzeta(\vx, \vF) > 0.$
Let $\va > 0$. Note that
    \begin{align*}
    \vJ_{a} &= \bigcup_{\vm \geq 1}\bigcup_{\vn \geq 1} \bigcap_{\vk \geq \vn} \bigcup_{\vq \in \Q^{\vd}}  \left(b(\vq, 1/\vk) \times \lbrace \vF \in \F: \vzeta(\vq, \vF) < \va - 1/\vm \rbrace \right),
    \end{align*}
    which leads to $\vJ_{\va} \in \B(\R^{\vd}) \otimes \sigma_\F$ for each $\va > 0$.
\end{proof}

\section{Proofs}
\label{EstSec}

\begin{proof}[Proof of Theorem~\ref{thm:consistency_tau1}]
    First, we rewrite $\tauest{}{N} $ as 
\begin{align*}
    \tauest{}{N} &=  \frac{1}{\H_{\vd-1}(\St_{\vXi}( \vWpO{})\cap \vWpIpp{}{\vN})} \int_{\St_\vXi(\vWpO))\cap \vWpIpp{}{N}} \vgamma_{\vXi}( \vx , \vWpO{\vl}) \, \mathcal{H}_{d-1}(\mathrm{d}\vx)\\
    &= \frac{\H_{\vd-1}(\vWpIp{\vN})}{\H_{\vd-1}(\St_{\vXi}( \vWpO{})\cap \vWpIpp{}{\vN})} \, \frac{1}{\H_{\vd-1}(\vWpIp{\vN})} \int_{\vWpIp{\vN}} \one_{o \in \St_{\vXi - \vx}( \vWpO{})}\vgamma_{\vXi - x}(o, \vWpO{\vl}) \, \mathcal{H}_{d-1}  (\mathrm{d}\vx).
\end{align*}
 Now, the individual ergodic theorem~\cite[Theorem 6.2]{chiu.2013} implies that the last expression converges almost surely to 
    \begin{align*}
     \frac{1}{\P(o \in \St_\vXi( \vWpO{\vl}))} \, \E[\vgamma_\vXi(o , \vWpO{\vl})\one_{o \in \St_\vXi( \vWpO{\vl})}] = \vtau_{}.
\end{align*}
This shows the strong consistency of $\tauest{}{N}.$
\end{proof}

Next, we prove Corollary~\ref{thm:consistency_tau2}.
\begin{proof}[Proof of Corollary~\ref{thm:consistency_tau2}]
    Since, $\tauest{}{N}$ equals $\tauest{\alpha}{N}$ given the event $\vRpp{\valpha}{N}$ we obtain that 
    $$|\tauest{\alpha}{N} - \vtau_{}| \le |\tauest{\alpha}{N} - \tauest{}{N}| + |\tauest{}{N} - \vtau_{}|  \le \one_{\vRpp{\valpha}{N}}|\tauest{\alpha}{N} - \tauest{}{N}| + |\tauest{}{N} - \vtau_{}|.$$
    Moreover, by Theorem \ref{thm:consistency_tau1} the second summand tends to 0 almost surely as $N \to \infty$. Since in cases (i),(ii) the first summand tends to 0 almost surely respectively in probability, we conclude the proof.
\end{proof}


\begin{proof}[Proof of Theorem~\ref{thm:consistency_rmax}]
    We define a family of stochastic processes $\lbrace \vU_N(r): N \ge 1 \rbrace$ with index set $[0, \infty)$ by
\begin{equation*}
\vU_{N}(r) = \frac{2\nu_{d}\left( \Psi_r(\vXi)\cap \vWpp{}{\vN}\right)}{\vN^{\vd-1}} - \frac{\nu_{d}(\vXi \cap \vWpp{}{\vN})}{\vN^{\vd-1}}.
\end{equation*}
    for all $\vr \geq 0, \vN \geq 1$. It is sufficient to show that  $\liminf_{\vN \to \infty}\vU_{\vN}(\vr) > 0$ for each $\vr < \rmax$ and $\limsup_{\vN \to \infty}\vU_{\vN}(\vr)<0$ for each $\vr > \rmax$.  First, let $\vr < \rmax$. Due to ergodicity as $N \rightarrow \infty$, the random variables $\vU_{N}(r)$ converge almost surely to 
    $2 \, \E[\nu_{d}\left( \Psi_r(\vXi)\cap \vWpp{}{1}\right)]- \E [\nu_{d}\left( \vXi\cap \vWpp{}{1}\right)],$
    which by assumption~\eqref{eq:assumption_consistency_rmax} is strictly positive. The case $\vr > \rmax$ follows verbatim.
\end{proof} 

\begin{proof}[Proof of Theorem~\ref{thm:consistency_rmin}]
	Note that $\vZ_{\vXi, \vr}$ is $\R^{d-1}$-stationary and ergodic as 
	$\vZ_{\vXi + \vx, \vr} = \vZ_{\vXi, \vr} + \vx$ for each $\vx \in \vWpI.$ Due to the ergodicity, the proof is analogous to the proof of Theorem \ref{thm:consistency_rmax}.	
\end{proof}

\begin{proof}[Proof of Corollary~\ref{cor:consistency_rmin_weaker_assumption}]
Part (i) follows directly from Condition~(\ref{eq:assumption_consistency_rmax}). It remains to prove assertion (ii). Let $0< \vr_1 < \rmin{}{} < \vr_2$ be arbitrary. Then,
\begin{enumerate}
    \item[(i$^\prime$)] $o \in \vXi^{\ominus \vr_1}$ implies that $B(o,\vr_1) \subset \vZ_{\vXi, r_1},$
    \item[(ii$^\prime$)] $o \not \in \Psi_{\vr_2}(\vXi)$ implies that $B(o,\vdelta) \cap \vZ_{\vXi, r_2} = \emptyset$ for some $\vdelta > 0$.
\end{enumerate}
Combining properties (i$^\prime$) and (ii$^\prime$) with Condition~(\ref{eq:weaker_assumption_consistency_rmin}) gives $\P(A) > 0$, where
$$\vA = \lbrace \nu_{d}((\vZ_{\vXi, \vr_1} \setminus \vZ_{\vXi, \vr_2}) \cap \vW_1) > 0 \rbrace.$$
Using $\vZ_{\vXi, \vr_2} \subset \vZ_{\vXi, \vr_1}$, we obtain that
\begin{align*}
	\E[\nu_d(Z_{\vXi, \vr_2} \cap \vW_1)] &= \P(A) \, \underbrace{\E [\nu_d(Z_{\vXi, \vr_2} \cap \vW_1) \mid A]}_{<\E [\nu_d(Z_{\vXi, \vr_1} \cap \vW_1) \mid A]} + (1 - \P(A)) \underbrace{\E [\nu_d(Z_{\vXi, \vr_2} \cap \vW_1) \mid A^c]}_{\leq \E [\nu_d(Z_{\vXi, \vr_1} \cap \vW_1) \mid A^c]},
\end{align*}
which is strictly smaller than $\E \nu_d(Z_{\vXi, \vr_1} \cap \vW_1)$ and therefore leads to the claim.
\end{proof}

\begin{proof}[Proof of Corollary~\ref{constCor}]
	Since the random variable $N_0$ is almost surely finite, the random set $\vZ_{\vXi \cap \vWpp{\alpha}{\vN},r}$ coincides with  $\vZ_{\vXi,r}$ for all $N \ge N_0$. Thus, we obtain assertion (i) by Theorem \ref{thm:consistency_rmin}. The assertion (ii) can be shown analogously. 
\end{proof}

\section{Analysis of edge effects}
\label{EdgeEffects}

The question arises to which extend plus sampling is required to estimate $\tau$ and $\rmin{}{}$ sufficiently precisely by the estimators given in (\ref{eq:tau_est_alpha}) and (\ref{eq:rmin_est_alpha}). In this section it is shown that 
the amount of required plus sampling is asymptotically negligible in comparison to the size of the sampling window for a certain type of a two-phase random-set model that has been considered previously in an application to materials science~\citep{n.2016}. Loosely speaking, this model results from a Voronoi mollification of a union of parametric proximity graphs, to be more precise, of beta-skeletons~\citep{kirkpatrick.1985}. Here we consider the special case, where the parameters of the beta-skeletons are chosen such that each of them coincides with the relative neighborhood graph.


First, in Section~\ref{defRngSec}, we provide a precise definition of this model. Then, in Section~\ref{tortRngSec}, we show that starting from a sampling window of side length $\vN\ge1$ the paths that are relevant for estimation of tortuosity are contained in an $\vN^\valpha$-environment with high probability, where $\valpha>0$ is an arbitrary positive number. The results obtained in Section~\ref{tortRngSec} are valid for an arbitrary dimension $d \geq 2.$ Finally, in Section~\ref{constRngSec}, we show that the issue of edge effects for constrictivity estimation can be translated into questions of an appropriately constructed continuum percolation model. We analyze this model first rigorously for very large and very small erosion radii, and then via simulation for intermediate values in Section~\ref{Numerics} to conclude that also for constrictivity estimation only an asymptotically negligible amount of plus sampling is required. Moreover, to ensure that constrictivity can be estimated strongly consistently in the multi-phase RNG model, a verification of Conditions (\ref{eq:assumption_consistency_rmax}) and (\ref{eq:weaker_assumption_consistency_rmin}) is necessary. As the corresponding proof is based on a rather technical construction, it is postponed to Appendix~\ref{sec:app_constrictivity}. Note that the results regarding the estimation of constrictivity are only valid for dimension $d=3$.

\subsection{Definition of a multi-phase RNG model}
\label{defRngSec}
In this section, we provide a mathematically precise definition of the multi-phase model under consideration as well as consistency results regarding the estimation of geodesic tortuosity and constrictivity. 

In materials science, microstructures, in which transport processes take place, exhibit a high degree of connectivity. Such highly connected structures can be represented by models based on connected random geometric graphs, such as the relative neighborhood graph (RNG)~\citep{jaromczyk.1992, n.2016}. Note that results regarding the lengths of shortest paths in the RNG itself have been discussed in~\citep{aldous.2010}.

 Loosely speaking, the phases are based on skeletons given by independent Poisson RNG, i.e., by RNG with vertices given by Poisson point processes. Then, we use a Voronoi mollification to associate the skeleton with a full-dimensional random closed set. In the following, we provide a more detailed description of both construction steps. First, we motivate the use of the RNG. 

The \emph{relative neighborhood graph} $\Rng(\vphi)$ is a graph on the locally finite vertex set $\vphi \subset \R^{d}$ where two nodes $\vx,\vy\in\vphi$ are connected by an edge if and only if there does not exist $\vz\in\vphi \setminus \lbrace x,y \rbrace$ such that $\max\{|\vx-\vz|,|\vy-\vz|\}\le |\vx-\vy|$. In order to model a $k$-phase material, $k\ge2$, we first build $k\ge1$ RNG based on independent homogeneous Poisson point processes $\vX\vda{1},\ldots,\vX\vda{k}$ with some intensities $\vla_1,\ldots,\vla_k>0.$ 

In a second step, we use a Voronoi mollification to construct for each of the $k$ graphs a full-dimensional random closed set representing the corresponding phase. More precisely, if $\vPhi=\{\vphi_i\}_{1 \le i\le k}$ is a collection of locally finite subsets of $\R^d$, we define the \emph{Voronoi mollification} of $\Rng(\vphi_i)$ with respect to $\vPhi$ by
$$\Vor(\vphi_i,\vPhi) = \{\vx\in\R^d:\,\dist(\vx,\Rng(\vphi_i))\le\inf_{1 \le j\le k}\dist(\vx,\Rng(\vphi_j))\},$$
where we use the notation $\dist(z, S) = \inf_{\vx \in S}|x- z|$ for each $S \subset \R^{d}$. That is,  $\Vor(\vphi_i,\vPhi)$ is the set of all points $\vx\in\R^d$ that are closer to the graph $\Rng(\vphi_i)$ than to any other of the graphs $\Rng(\vphi_1),\ldots,\Rng(\vphi_k)$, see Figure~\ref{modelFig}. Furthermore,  put $\vXi_i=\Vor(\vX\vda{i},\{\vX\vda{j}\}_{1\le j\le k})$ for each $i\in\{1,\ldots,k\}.$

\begin{theorem}
    \label{thm:edge}
    Let $\alpha \in (0,1)$ be arbitrary and $\vXi_1$ be the Voronoi mollification of a Poisson-RNG. 
    \begin{enumerate}
        \item There exists an almost surely finite random variable $\vN_0 \ge 1$ such that the event $\vRpp{\valpha}{\vN}$ occurs for all $\vN \ge \vN_0$. Moreover, $\E [\vgamma_{\vXi_1}(o, \vWpO{\vl}) \one_{o \in \St_{\vXi_1}( \vWpO{\vl})}] < \infty$.
        \item There exist  $0 < \vr_- < \rmin{}{} < \vr_+$ such that for every $\vr \notin (\vr_{-}, \vr_{+})$ there is an almost surely finite random variable $\vN_0 \ge 1$ such that the event $\vRminpp{\valpha}{\vN,r}$ occurs for all $\vN \ge \vN_0$. Moreover, 
        $\P( o \in \vXi_1^{\ominus \vr_1} \setminus \vPsi_{r_2}(\vXi_1)) > 0$ holds for all $0 < \vr_1 < \vr_2.$
    \end{enumerate}
\end{theorem}
We conjecture that the second part remains true for every $0 < \vr_- < \rmin{}{} < \vr_+$.

\begin{corollary}
    \label{cor:consEdge}
    Let $\alpha \in (0,1)$ be arbitrary and $\vXi_1$ be the Voronoi mollification of a Poisson-RNG. Then,
    \begin{enumerate}
        \item the estimator $\tauest{\valpha}{\vN}$ is strongly consistent as $N \rightarrow \infty.$
        \item there exist $0 < \vr_- < \rmin{}{} < \vr_+$ and an almost surely finite random variable $\vN_0$ such that the event $\vr_- \le  \rminestedge{}{}{\vN}{}\le \vr_+$ occurs for all $\vN \ge \vN_0$.
    \end{enumerate}
\end{corollary}
    If part (ii) of Theorem~\ref{thm:edge} was true for every $0 < \vr_- < \rmin{}{} < \vr_+$, then we could strengthen part (ii) of Corollary~\ref{cor:consEdge} to get strong consistency of $\rminestedge{}{}{\vN}{}$.
\begin{proof}[Proof of Corollary~\ref{cor:consEdge}]
    First, we argue that the estimators $\tauest{}{N}$ and  $\rminest{}{}{\vN}{}$ are consistent. Note that $\vXi_1$ is clearly $\R^{d-1}$-stationary. Moreover, due to the local nature of the RNG and mollification construction, one can also verify that it is ergodic \citep{daley.2008}. Hence, to apply Theorem~\ref{thm:consistency_tau1}, we only need to verify integrability of shortest-path lengths, which follows from part (i) in Theorem~\ref{thm:edge}. As explained in Corollary~\ref{thm:consistency_tau2}, we now conclude strong consistency of the edge-corrected estimator $\tauest{\valpha}{\vN}$.
    Similarly, to apply Theorems~\ref{thm:consistency_rmax} and~\ref{thm:consistency_rmin}, we use Corollary~\ref{constCor} and verify Condition~(\ref{eq:weaker_assumption_consistency_rmin}) in Appendix~\ref{sec:app_constrictivity}.  Alas, the conclusion of Theorem~\ref{thm:edge}, part (ii) is not strong enough to yield strong consistency of $\rminestedge{}{}{\vN}{}$. Nevertheless,  inspecting the proof of Corollary~\ref{constCor} reveals that the weaker condition is sufficient to deduce the weaker conclusion.
\end{proof}

\begin{figure}[!htpb]
\centering
\includegraphics[width = 0.5 \textwidth]{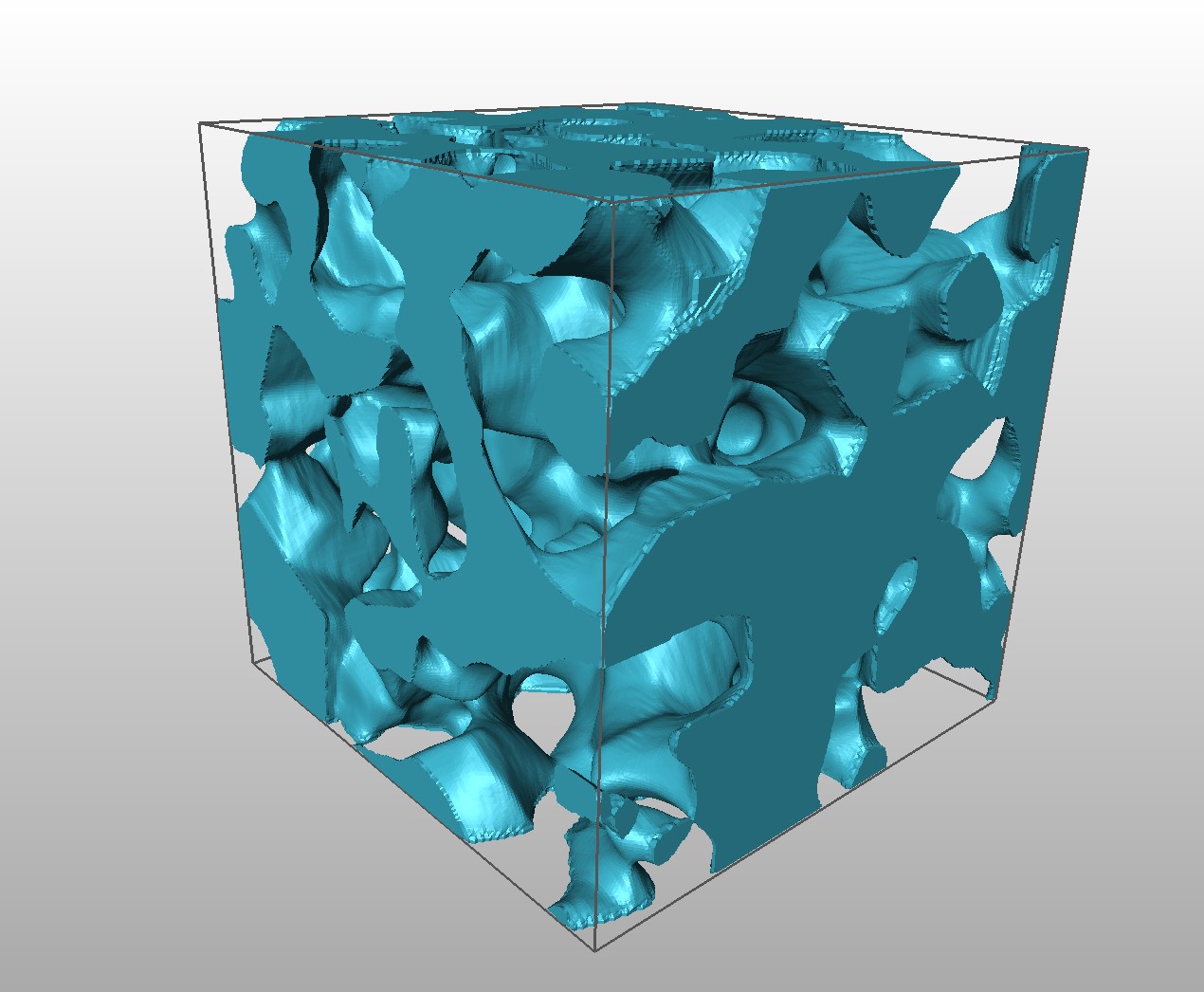}
\caption{Illustration of an RNG-based two-phase material model.}
\label{modelFig}
\end{figure}

\subsection{Geodesic tortuosity}
\label{tortRngSec}
Note that a strongly consistent estimation of mean geodesic tortuosity is possible in the multi-phase RNG model. More precisely, considering a sampling window of side length $\vN\ge1$, all shortest paths emanating from the bottom of the sampling window and ending at its top are contained in an $\vN^\valpha$-environment of the considered window, where $\valpha>0$ is an arbitrary positive number. This means that we will prove part (i) of Theorem~\ref{thm:edge}.



For this purpose, we first show that points in the Voronoi mollification of one of the $k$ relative neighborhood graphs are close to one of the edges of the graph. Then, we will see that with high probability (whp) there do not exist very long edges intersecting a given sampling window. Finally, to conclude the proof, we show that any pair of vertices of the relative neighborhood graph in an extended sampling window can be connected by a short path. We say that a family of events $\lbrace A_N, N \geq 1 \rbrace$ occurs \emph{whp} if $\P(A_N^c) \leq c_1 \, e^{-N^{c_2}}$ for some $c_1, c_2 > 0.$
\begin{lemma}
	\label{distToEdgeLem}
For every $\valpha>0$, with high probability, every point $x\in\vWpp{}{\vN}\cap\vXi_1$ can be connected within $\vXi_1$ to an edge of $\Rng(\vX\vda{1})$ by a line segment of length at most $N^\alpha$.
\end{lemma}
\begin{proof}
	Let $1 \leq i \leq k$ and subdivide the sampling window $\vWpp{}{\vN}$ into sub-boxes of side length $\vN^\alpha/\lceil\vN^{\valpha/2}\rceil$. Then, whp each of these sub-boxes contains at least one point from $\vX\vda{1}$. In particular, the distance of $x$ to the closest edge of $\Rng(\vX\vda{1})$ is at most $\vN^{\alpha}$ as required.
\end{proof}

\begin{lemma}
	\label{edgeLengthLem}
For every $\valpha>0$, with high probability, every edge of $\Rng(\vX\vda{1})$ intersecting $\vWpp{}{\vN}$ is of length at most $N^\alpha$.
\end{lemma}
\begin{proof}
	As in the proof of Lemma~\ref{distToEdgeLem}, we subdivide the sampling window $\vWpp{}{2\vN}$ into sub-boxes of side length $\vN^\alpha/\lceil\vN^{\valpha/2}\rceil$. Then, whp each of these sub-boxes contains at least one point from $\vX\vda{1}$. In particular, if $\vx,\vy$ are such that there is an edge between $\vx$ and $\vy$ in $\Rng(\vX\vda{1})$ such that the line segment $[x,y]$ intersects $\vW_N$, then $|\vx-\vy|\le N^\alpha$, as otherwise we could find $\vz\in \vX\vda{1}\cap\vW_{2N}$ with $\max\{|\vx-\vz|,|\vz-\vy|\}<|\vx-\vy|$.
\end{proof}

\begin{proof}[Proof of part (i) of Theorem~\ref{thm:edge}]
At first, we show that there exists an almost surely finite random variable $\vN_0 \ge 1$ such that the event $\vRpp{\valpha}{\vN}$ occurs for all $\vN \ge \vN_0$. The claim can be proven using an extension of the descending-chain approach from~\citep[Section 3]{aldous.2009}. Here, we recall from~\citep{daley.2005} that a sequence $\{\vx_1,\vx_2,\ldots,\vx_n\}$ of points in $\R^d$ forms a \emph{descending chain} if and only if the sequence $\{|\vx_i-\vx_{i+1}|\}_{1\le i\le n-1}$ of distances of subsequent elements is strictly decreasing. Moreover, we say that this chain is \emph{$b$-bounded} if $|\vx_1-\vx_2|\le b$. In~\citep[Lemma 10]{aldous.2009} it has been observed that there exists a tight connection between relative neighborhood graphs and descending chains. More precisely, if $\vphi\subset\R^d$ is a locally finite set and $r>0$, $\vx,\vy\in\vphi$ are such that $\vx$ and $\vy$ are not connected by a path in $\Rng(\vphi)\cap\vB_{r}(\vx)$, then there exists an $|\vx-\vy|$-bounded descending chain in $\vphi$ starting at $\vx$ and leaving $\vB_{r/2}(\vx)$.
In particular, if whp there does not exist an $\vN^{\alpha/(4d)}$-bounded descending chain starting in 
	$\vWpp{}{\vN}\oplus [-\vN^{\alpha/2},\vN^{\alpha/2}]$
	and consisting of more than $\vN^{\alpha/2}$ hops, then Lemmas~\ref{distToEdgeLem} and~\ref{edgeLengthLem} provide the desired short connection path. Now, the computations in~\citep[Section 3.2]{daley.2005} show that the probability for the existence of an $\vN^{\alpha/(4d)}$-bounded descending chain starting in $\vWpp{}{\vN}\oplus [-\vN^{\alpha/2},\vN^{\alpha/2}]$ and consisting of more than $K=\vN^{\alpha/2}$ hops is bounded from above by 
	$$\nu_d(\vWpp{}{\vN}\oplus [-\vN^{\alpha/2},\vN^{\alpha/2}])\frac{(\nu_d(B(o,1))\vN^{\alpha/4})^K}{K!}\le2\nu_d(\vWpp{}{\vN})\Big(\frac{3\nu_d(B(o,1))\vN^{\alpha/4}}{K}\Big)^K.$$
	Since the latter expression tends to $0$ as $\vN\to\infty$, we conclude the proof of the first assertion in part (i) of Theorem~\ref{thm:edge}. We now show that $\E [\vgamma_{\vXi_1}(o, \vWpO{\vl}) \one_{o \in \St_{\vXi_1}( \vWpO{\vl})}] < \infty$. Let $A_N$ denote the event that $o \notin \St_{\vXi_1}( \vWpO{\vl})$ or the shortest path from $o$ to $\vWpO{\vl}$ through $\vXi_1$ is contained in $ [-2N, 2N]^{d-1} \times [-N,1]$. We obtain 
	\begin{equation}\label{proof15_whp}
	\P(A^c_N) \leq c_1 e^{-N^{c_2}}
	\end{equation} 
	 for some $c_1, c_2 > 0$ as a consequence of Lemma~\ref{distToEdgeLem} and the first part of the proof. Define $N(r) = r^{1/(d+2)}.$ Then,
	\begin{align*}
	\E [\vgamma_{\vXi_1}(o, \vWpO{\vl}) \one_{o \in \St_{\vXi_1}( \vWpO{\vl})}] = &\int_{0}^{\infty} \P(\lbrace \vgamma_{\vXi_1}(o, \vWpO{\vl}) \one_{o \in \St_{\vXi_1}( \vWpO{\vl})} > r \rbrace \cap A_{N(r)}) \, \mathrm{d}r \\
	& + \int_{0}^{\infty} \P(\lbrace \vgamma_{\vXi_1}(o, \vWpO{\vl}) \one_{o \in \St_{\vXi_1}( \vWpO{\vl})} > r \rbrace \cap A^{c}_{N(r)}) \, \mathrm{d}r.
	\end{align*}
	Due to (\ref{proof15_whp}) the second integral on the right-hand side is finite. Thus it is sufficient to show that the first integral is finite. If the shortest path from $o$ to $\vWpO{\vl}$ through $\vXi_1$ is contained in $[-2N(r), 2N(r)]^{d-1}\times[-N(r),1]$, Lemma~\ref{edgeLengthLem} ensures that $\vgamma_{\vXi_1}(o, \vWpO{\vl}) \one_{o \in \St_{\vXi_1}( \vWpO{\vl})}$ is bounded from above by the sum of lengths of those edges, where at least one vertex is contained in $[-3N(r), 3N(r)]^d.$ Furthermore, the coordination number, i.e. the degree, of a vertex in the Poisson RNG is uniformly bounded by some constant $c_d$ \citep{jaromczyk.1992}. Thus
	\begin{align*}
		\int_{0}^{\infty} \P(\lbrace \vgamma_{\vXi_1}(o, \vWpO{\vl}) \one_{o \in \St_{\vXi_1}( \vWpO{\vl})} &> r \rbrace \cap A_{N(r)}) \, \mathrm{d}r \leq \\
		&\int_{0}^{\infty} \P(c_d 6\sqrt{d} N(r) \#(\vX^{(1)} \cap [-3N(r), 3N(r)]^d) \geq r) \, \mathrm{d}r,
	\end{align*}
	where $\#A$ denotes the number of elements contained in a discrete set $A \subset \R^3$. Using Markov's inequality, we obtain
	\begin{align*}
	\int_{0}^{\infty} &\P(\lbrace \vgamma_{\vXi_1}(o, \vWpO{\vl}) \one_{o \in \St_{\vXi_1}( \vWpO{\vl})} > r \rbrace \cap A_{N(r)}) \, \mathrm{d}r \leq 
	\int_{0}^{\infty} e^{-\frac{r^{(d+1)/(d+2)}}{c_d 6 \sqrt{d}}} \E e^{\#(\vX^{(1)} \cap [-3N(r), 3N(r)]^d)}\, \mathrm{d}r \\
	&= 	\int_{0}^{\infty} e^{-\frac{r^{(d+1)/(d+2)}}{c_d 6 \sqrt{d}}} e^{6^d r^{d/(d+2)}\vla_1(e - 1)}\, \mathrm{d}r \leq  c_4 \, \int_{0}^{\infty} e^{- c_5 r^{1/(d+2)}}\, \mathrm{d}r < \infty
	\end{align*}
	for some $c_4, c_5 > 0.$
\end{proof}

\subsection{Constrictivity}
\label{constRngSec}
Throughout this section, we assume that $\vd = 3.$ Recall that in Section~\ref{tortRngSec}, we have investigated shortest paths in the Poisson relative neighborhood graph. We have seen that only an asymptotically negligible amount of plus sampling is required to compute the estimators introduced in Section~\ref{estSec}. For the proof of part (ii) of Theorem~\ref{thm:edge}, which is given at the end of this section, the approach of Section~\ref{tortRngSec} has to be refined since the information coming from the relative neighborhood graph does not capture the bottleneck effects that are reflected in the notion of constrictivity. Although the edges of the relative neighborhood graph still play an important role, now we additionally need to take into account that due to closeness to a different phase some of these edges may no longer be available. This leads to a spatially correlated continuum percolation model that is investigated in the present section. In analogy to the relation between geodesic tortuosity and  the relative neighborhood graph, a deeper understanding of the geometry of paths in this continuum percolation model is critical to control the amount of plus sampling required to compute the estimator for $\rmin{}{}$ given in (\ref{eq:rmin_est_alpha}).

Next, we provide a rigorous definition of the spatially correlated percolation model described above. To simplify the presentation, we assume that $k=2$ and let $\vX\vda{1},\vX\vda{2}$ denote independent homogeneous Poisson point process with some intensities $\vla_1,\vla_2>0$, where $\vXi_i=\Vor(\vX\vda{i},\{\vX\vda{1},\vX\vda{2}\})$. For a constrictivity parameter $r\ge0$, we say that two points $\vx,\vy\in\vX\vda{i}$ are \emph{$r$-connected} if they are in the same connected component of $\vXi_i^{\ominus r}$. Moreover, we say that the set $\vXi_i$ is \emph{$r$-percolating} if $\vXi_i^{\ominus r}$ contains an unbounded connected component. Finally, we introduce the \emph{critical radius of percolation}
$\rc=\sup\{r\ge 0:\, \P(\vXi_1\text{ is $r$-percolating})>0\}$
as the supremum over all radii $r$ such that $\vXi_1$ is $r$-percolating with positive probability. Although connectivity of the Poisson relative neighborhood graph implies that $\rc\ge0$, it is not clear whether $\rc$ is non-trivial in the sense that $0<\rc<\infty$. In fact, the non-triviality of $\rc$ will be a corollary of the results established in this section.

However, for the problem of constrictivity estimation the non-triviality of $\rc$ by itself would be of limited use. For instance, in the sub-critical regime, i.e., if $r>\rc$, it is only known that there are no unbounded connected components. However, a priori, this does not exclude having two points that are spatially close, but are connected only by a very long path. It could also happen that two points are in different connected components, but each of these components is very large.
For the problem of minimizing the amount of plus sampling required for constrictivity estimation, such configurations are highly undesirable, since they lead to simulations to be carried out in large windows.

Therefore, in the present section, we provide quantitative results on the global geometry of paths thereby going beyond the mere existence of non-trivial sub- and super-critical regimes. More precisely, we first establish existence of a sub-critical regime with the additional property that observing long paths becomes exponentially unlikely. Second, we show that there exists a supercritical regime with a unique unbounded connected component. Inside this connected component, points can be connected by paths growing at most linearly in the Euclidean distance of the points whp. 

The percolation analysis described above allows us to control path lengths appearing in the constrictivity estimation if the constrictivity parameter $r$ is either large or close to $0$. In analogy to the classical model of Bernoulli percolation~\citep{grimmett.1999} it is natural to conjecture that this behavior prevails in fact for all values of $r$ with the possible exception of the critical threshold $\rc$. However, a rigorous proof requires substantially more refined arguments from percolation theory and is outside the scope of the present work. Instead, we provide ample numerical evidence supporting this conjecture.


\subsubsection{Subcritical regime}
To begin with, we analyze the subcritical regime and show that all connected components of $\vXir=\vXir_1$ intersecting a given large sampling window are much smaller than the size of the sampling window with high probability. In the following, let $\vQ_\vr(z) = [-r/2,r/2]^{3} + z$ for each $\vr > 0, \vz \in \Z^{3}.$ Below, we will use the abbreviation $\vQ_\vr = \vQ_\vr(o)$ for cubes centered at the origin.

\begin{theorem}
\label{subCritProp}
If $r>0$ is sufficiently large, then for every $\valpha > 0$ whp each connected component of $\vXir$ intersecting $\vWpp{}{\vN}$ has diameter at most $\vN^\valpha$.
\end{theorem}
\begin{proof}
    The assertion is shown by comparing continuum percolation of $\vXir$ to an appropriately constructed discrete-site percolation process. More precisely, for a site $z\in\Z^3$ we say that $z$ is \emph{$X^{(2)}$-closed} if and only if $\vX\vda{2}\cap\vQ_{r/3}(rz/3)\ne\es$. We also say that the associated cube $Q_{r/3}(rz/3)$ is $X^{(2)}$-closed. If a cube is not $X^{(2)}$-closed it is said to be \emph{$X^{(2)}$-open}.
	An elementary geometric argument shows that $\vXir$ does not intersect any $X^{(2)}$-closed cubes. In particular, the connected component of $\vXir$ at any given point in $\vXir\cap\vWpp{}{\vN}$ is contained in the connected component of $X^{(2)}$-open cubes containing that point. Moreover, since $\vX\vda{2}$ is a Poisson point process, different sites are $X^{(2)}$-closed independently of each other and the probability for a given site to be $X^{(2)}$-closed is arbitrarily close to 1 provided that $r$ is sufficiently large. Hence, we may apply the exponential decay of cluster sizes in Bernoulli percolation~\citep[Theorem 6.75]{grimmett.1999}. In particular, the probability that there exists an $X^{(2)}$-open cube $\vQ_{r/3}(rz/3)$ intersecting $\vWpp{}{\vN}$ whose connected component is of diameter at least $\vN^\alpha$ tends to $0$ exponentially fast as $\vN\to\infty$. 
\end{proof}

\subsubsection{Supercritical regime}
Next, we investigate the supercritical regime, i.e., the scenario when the constrictivity parameter $r$ is close to $0$. Our main result shows that $\vXir$ splits up into two domains exhibiting fundamentally different properties. On the one hand, there exists a unique unbounded connected component $\Cinf\subset\vXir$. This connected component features good connectivity properties. Indeed, any two points in this component can be connected by a path whose diameter grows at most linearly in the Euclidean distance. On the other hand, the diameters of finite connected components exhibit exponentially decaying tail probabilities. 

\begin{theorem}
	\label{supCritProp}
	If $r>0$ is sufficiently small, then with probability 1, there exists a unique unbounded connected component $\mc{C}_\infty$ in $\vXir$. Moreover, for every $\valpha>0$ it holds whp that if $x\in\vWpp{}{\vN}$ is such that 
\begin{enumerate}
	\item $x\not\in\mc{C}_\infty$, then the connected component of $\vXir$ containing $x$ has diameter at most $\vN^\alpha$. 
	\item $x\in\mc{C}_\infty$, then $x$ can be connected in $\vXir$ to $\vWpO{\vl}$ by a path of diameter at most $\vN^\alpha$.
\end{enumerate}
\end{theorem}

The proof of Theorem~\ref{supCritProp} is based on a classical coarse-graining argument comparing percolation properties of $\vXir$ to those in a suitably discretized model. More precisely, for $\vL\ge1$ and $r>0$ we say that a site $z\in\Z^3$ is \emph{$(\vL,r)$-good} if the intersection $\vXir\cap\vQ_{\vL}(\vL z)$
\begin{enumerate}
	\item[{\bf (C)}] is contained in a connected component of $\vXir\cap\vQ_{3\vL}(\vL z)$ (connectivity condition), and 
	\item[{\bf (O)}] features non-empty intersections with each of the boundary faces of the cube $\vQ_{\vL}(\vL z)$ (omnidirectionality condition).

\end{enumerate}
A site $z\in\Z^3$, which is not $(\vL,r)$-good, is called \emph{$(\vL,r)$-bad}. For $z\in\Z^3$ we let $\vC(z)$ denote the $*$-connected component of $(\vL,r)$-bad sites which contains $z$. Following~\citep{antal.1996}, a \emph{$*$-connected} component is a connected component of a graph whose vertex set is a subset of $\Z^3$ and edges consisting of pairs of vertices $z,z'\in\Z^3$ satisfying $|z-z'|_\infty\le1$. Here $|\cdot|_\infty$ denotes the maximum norm in the Euclidean space. Our proof of Theorem~\ref{supCritProp} is based on the following exponential moment bound for $\vC(z)$, whose proof is deferred to later parts of this section.
\begin{proposition}
	\label{expMomBoundProp}
	There exist $\vL\ge 1$ and $r>0$ such that $\E\exp(64\#\vC(o))<\infty$.
\end{proposition}
To simplify terminology, we call a site bad if it is $(\vL,r)$-bad with the parameters $\vL$, $r$ appearing in Proposition~\ref{expMomBoundProp}. As an important corollary, we note that Proposition~\ref{expMomBoundProp} brings us into a position where we can apply the classical Peierls argument \citep{grimmett.1999}, in the sense that there exists only a finite number of bad $*$-connected components separating the origin from infinity. That is, there exists only a finite number of $*$-connected bad sets $A\subset\Z^3$ such that $o$ is not contained in the unbounded connected component of $\Z^3\setminus A$. More precisely, writing $E_k$ for the event that the origin is separated from infinity by a $*$-connected bad set of diameter at least $k$, we will show that the probability of $E_k$ decays exponentially in $k$.

\begin{corollary}
	\label{peiCor}
There exist $\vL\ge 1$ and $r>0$ such that 
	$$\limsup_{k\to\infty}k^{-1}\log\P(E_k)<0.$$
\end{corollary}
\begin{proof}
	First, if $z$ is an arbitrary site in $A$, then $A$ can separate the origin from infinity only if $\#A\ge|z|_\infty$. Therefore, using that for $m\ge1$ the number of $*$-connected sets having $m$ elements and containing the origin is at most $2^{27m}$~\citep[Lemma 9.3]{penrose.2003}, the expected number of bad sets of diameter at least $k$ separating the origin from infinity is bounded from above by
	\begin{align*}
		\sum_{z\in\Z^3}\sum_{\substack{A\ni z\text{ is $*$-connected}\\ \#A\ge\max\{k,|z|_\infty\}}}\P(\vC(z)=A)&\le\sum_{z\in\Z^3}\sum_{m\ge\max\{k,|z|_\infty\}}2^{27m}\P(\#\vC(o)\ge m)\\
		&\le\sum_{z\in\Z^3}\sum_{m\ge\max\{k,|z|_\infty\}}\exp(m(27\log 2-64))\E\exp(64\#\vC(o)).
	\end{align*}
	Hence, Proposition~\ref{expMomBoundProp} implies that 
	$$\limsup_{k\to\infty}k^{-1}\log\P(E_k)\le\tfrac{1}{2}(27\log 2-64),$$
	as required.
\end{proof}

Next, we deduce Theorem~\ref{supCritProp} from Proposition~\ref{expMomBoundProp} and Corollary~\ref{peiCor}.

\begin{proof}[Proof of Theorem~\ref{supCritProp}]
    First, combining Corollary~\ref{peiCor} with the Borel-Cantelli lemma, we conclude that almost surely there are only finitely many bad $*$-connected components separating the origin from infinity. In particular, the external outer boundary of the last such component is part of an infinite $*$-connected component. Moreover, the probability that there exists a separating $*$-connected component at distance at least $k$ from the origin decays exponentially in $k$. In particular, this establishes the existence of an unbounded connected component and the claim on the diameter of the bounded connected components. To establish uniqueness of the unbounded connected component, we first note that if $\gamma$ is a nearest-neighbor path in $\Z^3$ consisting of $k$ steps, then Proposition~\ref{expMomBoundProp} implies that the probability for the external outer boundary of the union of bad $*$-connected components to intersect $\gamma$ is of diameter larger than $64k$ decays to 0 exponentially fast in $k$. In particular, this verifies both assertion (ii) and that $\vXir$ has a unique unbounded connected component. 
\end{proof}

The remainder of this section is devoted to the proof of Proposition~\ref{expMomBoundProp}. To establish the desired bound on the exponential moment, we combine the stabilization techniques from~\citep{penrose.2001} with the lattice construction used in~\citep{antal.1996}. We also refer the reader to~\citep{aldous.2009}, where a similar strategy has been applied successfully to a planar relative neighborhood graph. Despite the conceptual similarities, for the convenience of the reader we present a detailed proof. As a first step, we remove the long-range dependencies inherent to the definition of bad sites by introducing a robust variant of the notion of a bad site. This variant has the advantage of exhibiting a finite range of dependence.

First, loosely speaking, every point in $\vXir\cap\vQ_{\vL}(\vL z)$ should be connected by a path in $\vXir\cap\vQ_{2\vL}(\vL z)$ to an edge of the graph $\Rng(\vX\vda{1})\cap\vQ_{2\vL}(\vL z)$. Second, the relative neighborhood graph $\Rng(\vX\vda{1})\cap\vQ_{2\vL}(\vL z)$ should intersect every face of the cube $\vQ_{\vL}(\vL z)$ and be contained in a connected component of $\vXir\cap\vQ_{3\vL}(\vL z)$. Third, the configuration $\vXir\cap\vQ_{3\vL}(\vL z)$ should be stable in the sense that it does not change if the Poisson point processes are altered away from the larger cube $\vQ_{7\vL}(\vL z)$. This is achieved by making sure that close to any point of $\vQ_{7\vL}(\vL z)$ there exist Poisson points from both $\vX\vda{1}$ and $\vX\vda{2}$. More precisely, we say that a site $z \in \Z^3$ is \emph{robustly $(\vL,r)$-good} if
\begin{enumerate}
	\item[{\bf (C')}] $\Rng(\vX\vda{1})\cap\vQ_{2\vL}(\vL z)$ and $\vXir\cap\vQ_{\vL}(\vL z)$ are contained in a common connected component of $\vXir\cap\vQ_{3\vL}(\vL z)$ (connectivity condition),
	\item[{\bf (O')}] $\Rng(\vX\vda{1})\cap\vQ_{2\vL}(\vL z)$ intersects every face of the cube $\vQ_{\vL}(\vL z)$ (omnidirectionality condition), and
	\item[{\bf (S')}] $\vX\vda{i}\cap\vQ_{\vL/4}(\vL z'/4)\ne\es$ for every $i\in\{1,2\}$ and $z'\in \Z^3\cap\vQ_{28}(z)$ (stability condition).
\end{enumerate}
Since every robustly $(\vL,r)$-good site is also $(\vL,r)$-good, it remains to establish the exponential moment bound in Proposition~\ref{expMomBoundProp} for the $*$-connected component of not robustly $(\vL,r)$-good sites. Harnessing the dependent-percolation techniques from~\citep{liggett.1997}, the latter assertion hinges on the property that robustly $(\vL,r)$-good sites (i) become likely for large $L$ and small $r$, and (ii) exhibit a finite range of dependence.

\begin{lemma}
	\label{finDepLem}
	Let $L\ge1$, $r>0$ and assume that the stability condition \emph{{\bf (S')}} is satisfied at the origin. Then, the random set $\vXir\cap\vQ_{3\vL}$ is not changed by modifications of the point process $\vX\vda{1}\cup\vX\vda{2}$ outside $\vQ_{7\vL}$.
\end{lemma}

\begin{lemma}
	\label{highProbLem}
	It holds that 
	$\lim_{\vL\to\infty}\lim_{r\to0}\P(o\text{ is $(\vL,r)$-good})=1.$
\end{lemma}

Before we prove Lemmas~\ref{finDepLem} and~\ref{highProbLem}, we show how they can be used to establish Proposition~\ref{expMomBoundProp}.
\begin{proof}[Proof of Proposition~\ref{expMomBoundProp}]
	By Lemma~\ref{finDepLem} the process of robustly $(\vL,r)$-good sites is a stationary $7$-dependent site percolation process. Hence, Lemma~\ref{highProbLem} and~\citep[Theorem 0.0]{liggett.1997} imply that it stochastically dominates a supercritical Bernoulli site percolation process provided that $\vL$ and $r$ are chosen sufficiently large and small, respectively. As suitable choices of $\vL$ and $r$ allow to push the site activation probability of this Bernoulli percolation process arbitrarily close to $1$, the desired exponential moment bound becomes a consequence of a standard first-moment bound in percolation theory~\citep[Theorem 1.10]{grimmett.1999}.
\end{proof}
It remains to prove Lemmas~\ref{finDepLem} and~\ref{highProbLem}. To begin with, we show Lemma~\ref{finDepLem} which requires only elementary geometric arguments.
\begin{proof}[Proof of Lemma~\ref{finDepLem}]
	As a first step, we show that $\Rng(X^{(i)})\cap\vQ_{5\vL}$ is left unchanged by modifications of $\vX\vda{i}$ outside $\vQ_{7\vL}$. Otherwise, by the definition of the relative neighborhood graph, there would exist a ball of diameter $L$ intersecting $\vQ_{5\vL}$ that does not contain any points from $\vX\vda{i}$. This would yield a contradiction to the stability condition \emph{{\bf (S')}}. Moreover, by condition \emph{{\bf (S')}}, the distance of any point in $\vQ_{3\vL}$ to the nearest point on $\Rng(\vX\vda{i})$ is at most $L/2$. Therefore, the random set $\vXir\cap\vQ_{3\vL}$ is not changed by modifying $\vX\vda{i}$ outside $\vQ_{7\vL}$.
\end{proof}
Next, we show that conditions {\bf (O')} and {\bf (S')} are fulfilled whp. 
\begin{proof}[Proof of Lemma~\ref{highProbLem}, conditions {\bf (O')} and {\bf (S')}]
    First, since every robustly $(\vL,r)$-good site is also $(\vL,r)$-good, we obtain that
    $$\lim_{\vL\to\infty}\lim_{r\to0}\P(o\text{ is robustly $(\vL,r)$-good}) \le \lim_{\vL\to\infty}\lim_{r\to0}\P(o\text{ is $(\vL,r)$-good}).$$
    Hence, it remains to show that conditions {\bf (S')}, {\bf (O')} and {\bf (C')} are fulfilled with a probability tending to 1 as $L \to \infty$ and $r \to 0$. The exponential decay of the Poisson void probabilities shows that both, the omnidirectionality and the stability condition are satisfied whp as $\vL\to\infty$. As we explain below, proving the corresponding property for condition {\bf (C')} is more involved and is therefore postponed to Appendix~\ref{app:condition_c}. 
\end{proof}
It remains to check that condition {\bf (C')} is fulfilled with a probability tending to 1 as $L \to \infty$ and $r \to 0$, which is surprisingly complicated. First, a standard descending-chain argument~\citep[Proposition 9]{aldous.2009} shows that whp the graph $\Rng(\vX\vda{1})\cap\vQ_{2\vL}$ is contained in a connected component of $\Rng(\vX\vda{1})\cap\vQ_{3\vL}$. Moreover, by choosing $r>0$ sufficiently small, the probability that $\Rng(\vX\vda{1})\cap\vQ_{3\vL}\subset\vXir$ can also be pushed arbitrarily close to 1. It remains, to consider points of $\vXir\cap\vQ_{\vL}$ that do not lie on $\Rng(\vX\vda{1})$. 

A first attempt would be to connect any such point to the closest points on $\Rng(\vX\vda{1})$ by a line segment directly. However, due to the non-convex nature of the eroded cells in $\vXir$ this line segment need not be contained in $\vXir\cap\vQ_{\vL}$. Certainly, it is intuitively plausible that for small $r$ it should not be difficult to circumvent these obstacles. Nevertheless, making this line of argumentation rigorous requires a refined analysis of the deterministic cell geometry of $\vXir$. As this analysis is quite different from the probabilistic arguments appearing in the rest of the paper, we defer it to Appendix \ref{app:condition_c}.

\begin{proof}[Proof of part (ii) of Theorem~\ref{thm:edge}]
	Theorem~\ref{subCritProp} ensures the existence of $\vr_+$, while Theorem~\ref{supCritProp} ensures the existence of $\vr_-$ with the required properties. The proof of  $\P( o \in \vXi_1^{\ominus \vr_1} \setminus \vPsi_{r_2}(\vXi_1)) > 0$ for all $0 < \vr_1 < \vr_2,$ is postponed to Appendix~\ref{sec:app_constrictivity}. 
\end{proof}

\section{Numerical simulations}
\label{Numerics}
In applications, e.g. in computational materials science, the geometry of individual phases within the microstructures of materials under consideration is investigated based on 3D image data. That is, the information about the geometry is available on a discrete grid. Although the estimators introduced in Section~\ref{DefSec} require information in continuous space, they can be approximated from discrete data. The numerical simulation study, performed in this section, shows that also for the approximation from discrete data the values of the estimators remain unchanged when the sampling window is large enough. The results are obtained based on realizations of the multi-phase RNG model, which are discretized on $\mathbb{Z}^3$. Note that the above mentioned stabilization of estimators with respect to increasing sampling windows is also valid for the estimation of $\rmin{}{}$, although the required conditions for strong consistency in Corollary~\ref{constCor} are difficult to check and we verified them only partially for the multi-phase RNG model in Section~\ref{EdgeEffects}.   

To begin with, the random closed set $\vXi = \Vor(X^{(1)}, \lbrace X^{(1)}, X^{(2)} \rbrace)$ is simulated in the sampling window $[-750 , 750]^2 \times [-150, 550]$, where $X^{(1)}$ and $X^{(2)}$ are two independent homogeneous Poisson point processes in $\R^{3}$ with intensities $\vla_{1}, \vla_{2} > 0.$ In the following, we consider tortuosity and constrictivity of $\vXi$ with respect to transportation paths from $\vWpI$ to $l \vWpO{}$ with $l = 400,$ cf. Section \ref{DefTauSec}. We simulate one realization for each vector $\vla=(\vla_{1}, \vla_{2})$ of intensities contained in $\lbrace 3 \cdot 10^{-5}, 5 \cdot 10^{-5} \rbrace^{2}.$ In the first step, the Poisson point processes are simulated as described in \citep{moeller.2004}. In order to avoid edge effects the point processes are simulated in the enlarged sampling window $[-800 \times 800]^2 \times [-200, 600].$ In the second step, the relative neighborhood graphs are computed. In the third step, for each $z \in \mathbb{Z}^{3}$ we check if $z \in \vXi.$ For that, the edges of $\Rng(X^{(1)})$ and $\Rng(X^{(2)})$ are discretized and the distance of each point  $z \in \mathbb{Z}^{3}$ to both graphs is computed by the algorithm described in \citep{felzenszwalb.2012}.

\begin{figure}
		\includegraphics[width = 0.48\textwidth]{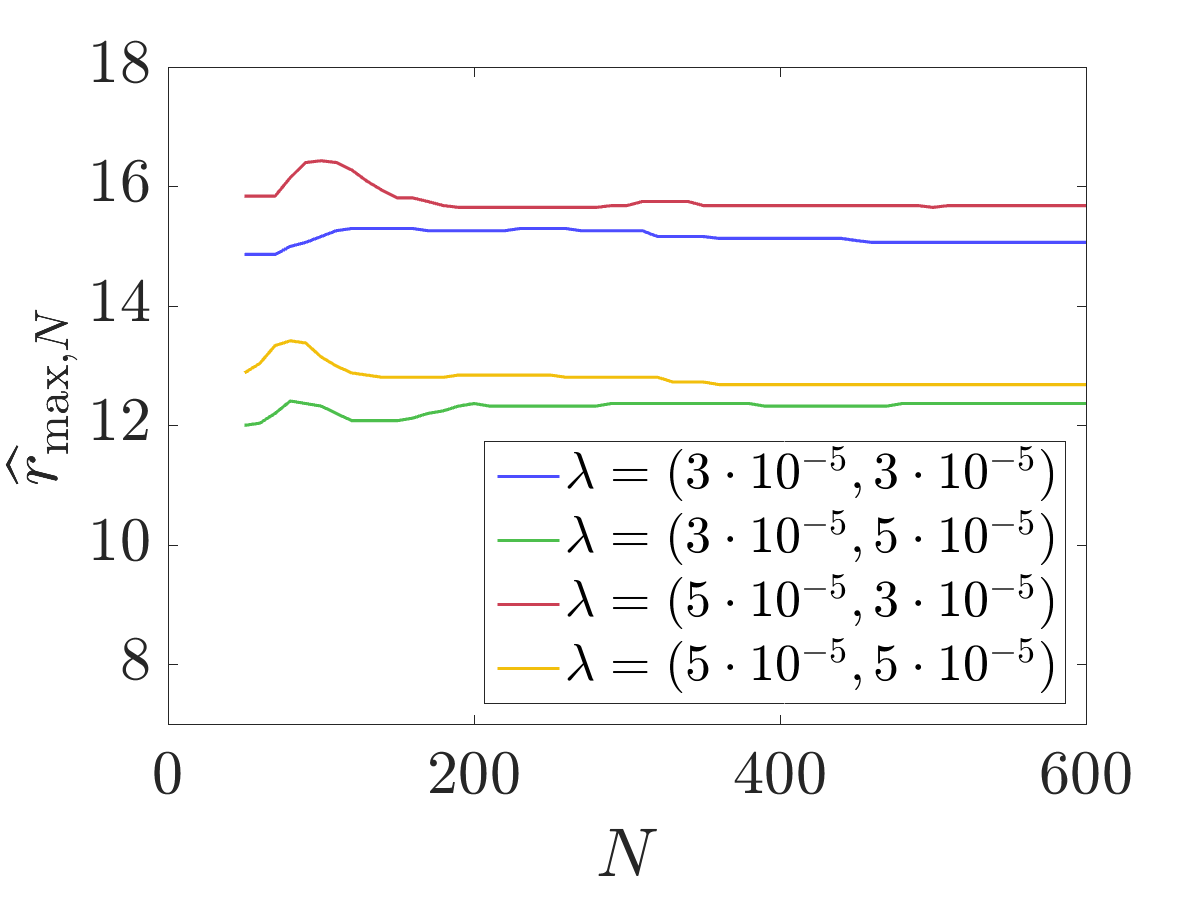}
		\includegraphics[width = 0.48\textwidth]{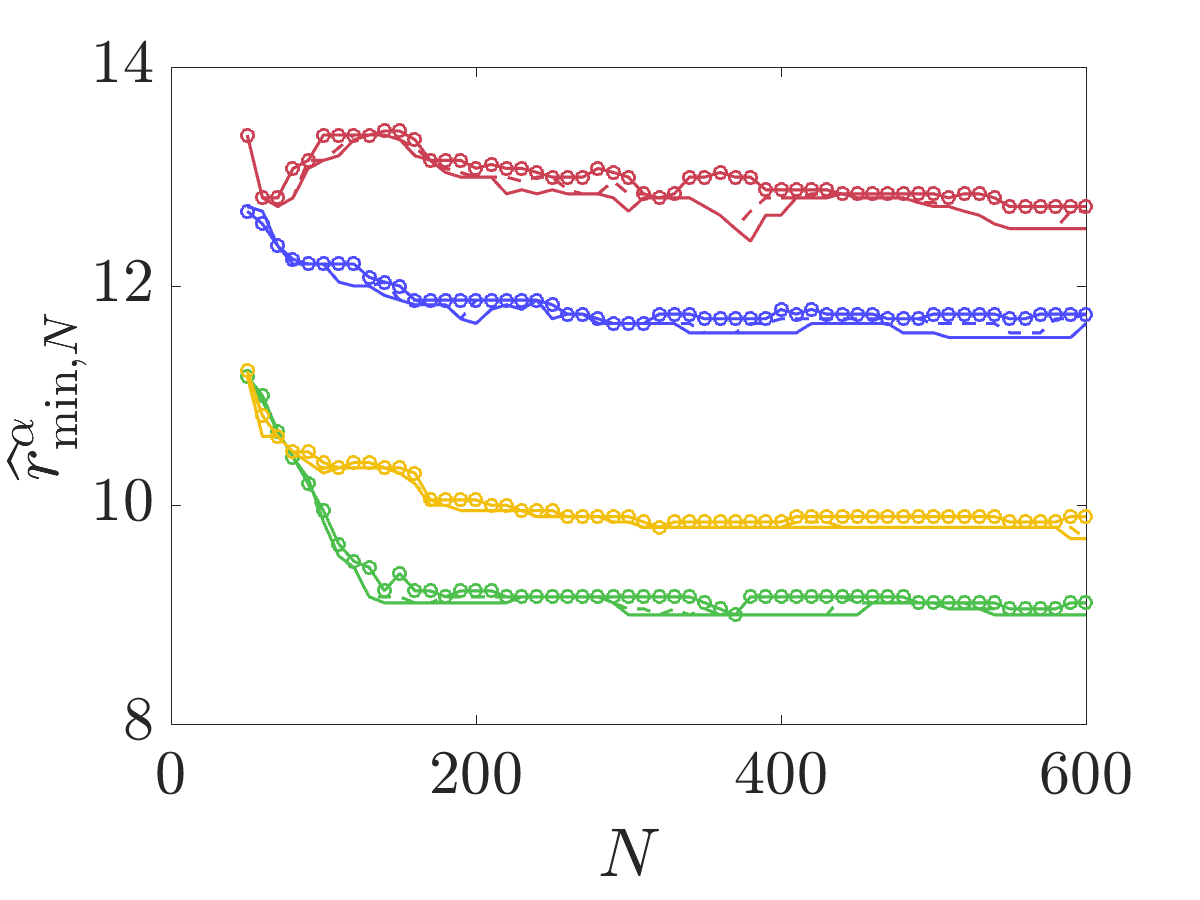}\\
        \includegraphics[width = 0.48\textwidth]{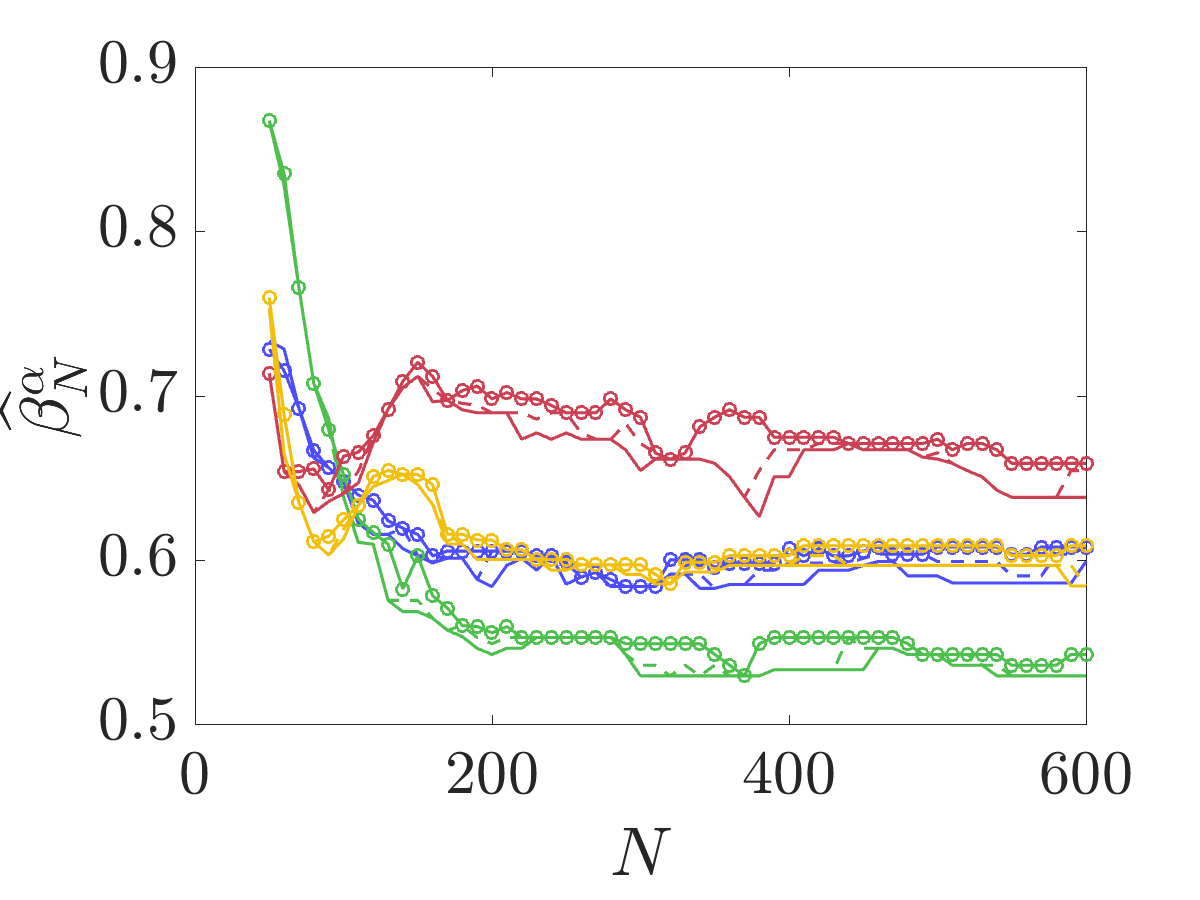}
		\includegraphics[width = 0.48\textwidth]{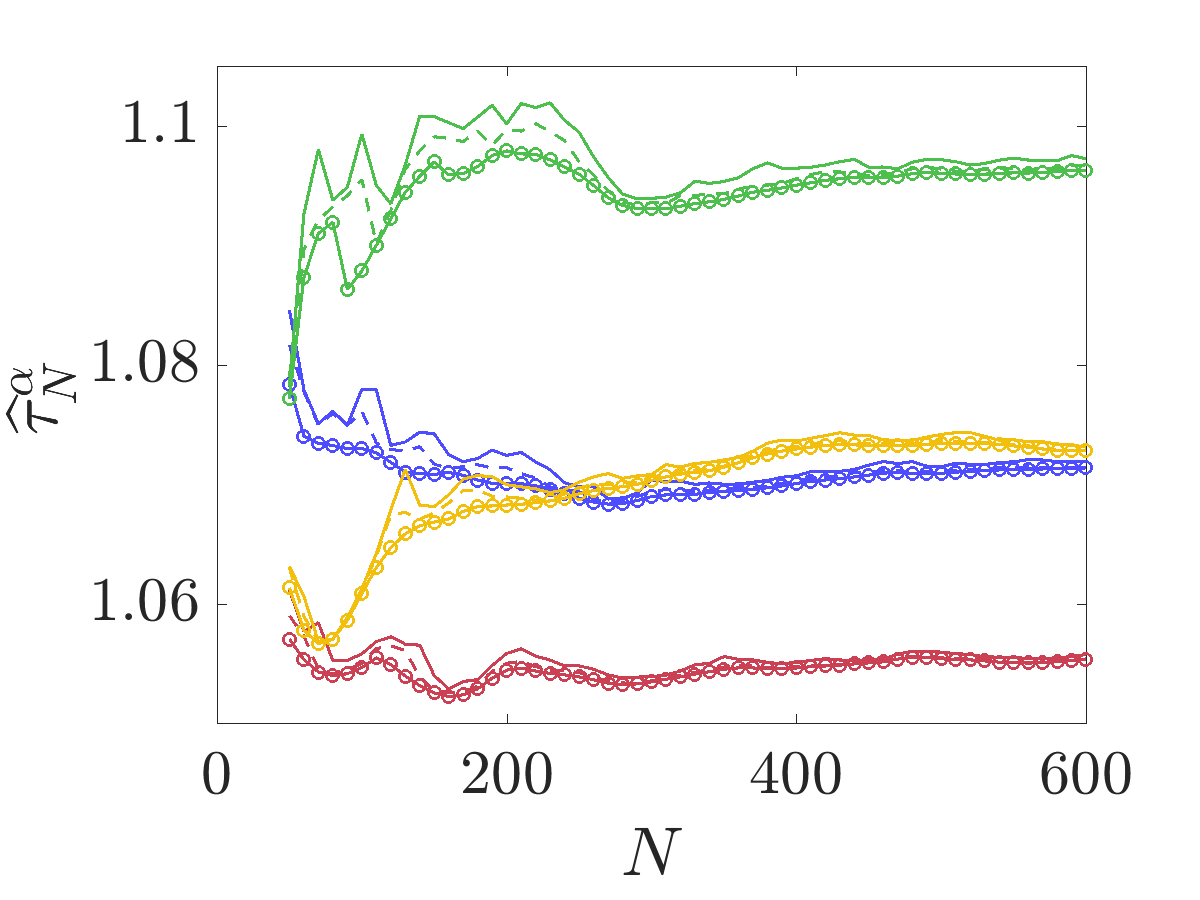}
		\caption{Estimation of mean geodesic tortuosity and constrictivity. In the plots regarding $\rminestedge{N}{}{}{}$, $\widehat{\beta}_{N}^{\alpha}$ and $\tauest{\alpha}{N}$, the straight line represents the estimator for $\alpha=1/4$, the dashed line for $\alpha=1/2$ and the line with circles for $\alpha=3/4$.}\label{fig:simulation_study}
\end{figure}

Based on the realizations of the multi-phase RNG model, we compute discrete approximations of the estimators $\tauest{\valpha}{10}, \tauest{\valpha}{20}, \ldots, \tauest{\valpha}{600}$ for each $\alpha \in \lbrace 1/4, 1/2, 3/4\rbrace.$ Before we give the definition
of  discrete approximation of $\tauest{\alpha}{N}$, denoted by $\discr(\tauest{\alpha}{N}),$ some additional notation is introduced. For each pair of points $x,y \in \mathbb{Z}^{3}$ we say that $x \sim y$ if $|x - y|_{\infty} \leq 1.$ Furthermore, we define the set 
\begin{align*}
\discr(\mathcal{C}_{\vXi \cap \Z^{3}}(\vWpO{})) = \lbrace x \in &\vXi \cap \Z^{3} \cap \vWpI: \text{for some } \\ & x = x_{1}, \ldots, x_n\in \vXi \cap \Z^{3},  \text{ with }  x_{1} \sim \cdots \sim x_n, x_n \in \vWpO{} \rbrace.
\end{align*}
This notation allows for the definition of
\begin{equation}
\discr(\tauest{\alpha}{N}) = \frac{1}{\# \discr(\mathcal{C}_{\vXi \cap \Z^{3}}(\vWpO{}))} \, \sum_{x \in \discr(\mathcal{C}_{\vXi \cap \Z^{3}}(\vWpO{}))}  \min_{y \in \vWpO{}} \, \min_{\substack{n \geq 2, \,\vx = \vx_{1} \sim \cdots \sim \vx_{\vn} = \vy  \\ \vx_{1}, \ldots, x_{n} \in \vXi \cap  \Z^{3} \cap \vWpp{\valpha}{\vN} }} \, \sum_{i=1}^{n-1} | x_{i} - x_{i+1} |.
\end{equation}
According to \citep{davis.2017} we may expect that the discrete estimator leads to a good approximation of $\tauest{\alpha}{N}$. 

In order to estimate constrictivity from discrete data, consider a discrete version $\discr(\vXi)=\vXi \cap \mathbb{Z}^{3}$ of $\vXi$.
Based on $\discr(\vXi)$, a discrete version $\discr(\vPsi_\vr(\vXi))$ of $\vPsi_\vr(\vXi)$ is computed, which allows for the estimation of $\rmax$. In order to estimate $\rmin{}{}$ the connected components of $\discr(\vPsi_\vr(\vXi))$ are computed with respect to the 26-neighborhood by the aid of the Hoshen-Kopelman algorithm~\citep{hoshen.1976}. 
    
The results of the simulation study are visualized in Figure~\ref{fig:simulation_study}. For the considered parameter constellations, the following can be observed. First, if the size of the sampling window is large enough, the estimators for $\tau, \rmax$ and $\rmin{}{}$ do not change significantly after a further enlargement of the sampling window. Note that this holds also for $\rmin{}{}$, for which the strong consistency of the estimator has not been theoretically established in Section~\ref{EdgeEffects}. Moreover, it can be said that the estimation of $\rmax$ becomes stable for smaller sampling windows compared to $\tau$ and $\rmin{}{}.$ Second, the estimators of both, $\rmin{}{}$ and $\tau$, which are determined based on the enlarged sampling window $\vWpp{\valpha}{\vN}$, do not vary for the considered choices of $\alpha$ provided that $\vN$ is large enough. This is in good accordance with the results from Corollary~\ref{cor:consEdge}.

\appendix
\section{Verification of Condition \eqref{eq:weaker_assumption_consistency_rmin} for Voronoi Mollifications}\label{sec:app_constrictivity}

We show that Condition \eqref{eq:weaker_assumption_consistency_rmin} is fulfilled if $\Xi$ is a Voronoi mollification in $\R^3$.

\begin{theorem}\label{thm:granulometry_decreasing}
	Let $k \geq 1$ and $\vX^{(1)}, \ldots, \vX^{(k)}$ be independent homogeneous Poisson point processes in $\R^3$ with intensities $\vla_{1}, \ldots, \vla_{k} > 0$. Let $\vXi = \Vor(\vX^{(1)}, \lbrace \vX^{(1)}, \ldots, \vX^{(k)} \rbrace)$ and $0< r_{1} < r_{2}$ be arbitrary. Then, 
	$\P( o \in \vXi^{\ominus \vr_1} \setminus \vPsi_{r_2}(\vXi)) > 0.$
\end{theorem}

To prove Theorem \ref{thm:granulometry_decreasing}, some lemmas concerning the RNG in $\R^3$ are proven. In the following, we consider the finite point pattern 
$$\vphi= \lbrace \vx_{1}, \ldots, \vx_{8} \rbrace = \left \lbrace \left((-1)^{i}\va, (-1)^{j}\va, (-1)^{k}\va \right): i,j,k \in \lbrace 0, 1\rbrace \right\rbrace \subset \R^3,$$ 
where $\va = (r_1 + r_2)/4$. Additionally to $\vphi$, we consider a finite point pattern $\vphi(\varepsilon) = \lbrace \vy_{1}, \ldots, \vy_{8} \rbrace \subset \R^{3}, $ where $y_i \in B(x_i, \varepsilon)$ for each $1 \leq i\leq 8$.

\begin{lemma}\label{lem:edges_in_rng_on_cube}
	Let $1 \leq i < j \leq 8$ and let $\varepsilon < a (\sqrt{3} - \sqrt{2}) / 2$ . Then, $\vy_{i}$ and $\vy_{j}$ are connected by an edge in $\Rng(\vphi(\varepsilon))$ if and only if $| \vx_{i} - \vx_{j}| = 2a.$
\end{lemma} 
\begin{proof}
	Note that $|\vx_{i} - \vx_{j}| \in \lbrace 2a, 2\sqrt{2}a, 2\sqrt{3}a \rbrace.$  Let $|\vx_{i} - \vx_{j}| = 2a.$ Then it is 
	$$ \max \lbrace |\vy_{i} - \vz|, |\vy_{j} - z| \rbrace \geq 2 \sqrt{2} a - 2 \varepsilon > 2\sqrt{2}a - a(\sqrt{2} - 1) = 2a + (\sqrt{2} - 1)a > 2a + 2\varepsilon > |\vy_{i} - \vy_{j}|,$$ for each $\vz \in \vphi(\varepsilon)\setminus \lbrace \vy_{i}, \vy_{j} \rbrace.$ Thus $\vy_{i}$ and $\vy_{j}$ are connected by an edge in $\Rng(\vphi(\varepsilon))$.\\
	Let $|\vx_{i} - \vx_{j}| = 2\sqrt{2}a.$ Then, there exists $\vz \in \varphi(\varepsilon) \setminus \lbrace \vy_{i}, \vy_{j} \rbrace$ such that 
	$$ \max \lbrace |\vy_{i} - \vz|, |\vy_{j} - z| \rbrace \leq 2a + 2\varepsilon < 2\sqrt{2}a - 2 \varepsilon = |\vy_{i} - \vy_{j}|.$$ Thus $\vy_{i}$ and $\vy_{j}$ are not connected by an edge in $\Rng(\vphi(\varepsilon))$.\\
	Let $|\vx_{i} - \vx_{j}| = \sqrt{3}a.$ Then, there exists $\vz \in \varphi(\varepsilon) \setminus \lbrace \vy_{i}, \vy_{j} \rbrace$ such that 
	$$ \max \lbrace |\vy_{i} - \vz|, |\vy_{j} - z| \rbrace \leq 2\sqrt{2}a + 2 \varepsilon < 2 \sqrt{3}a - a(\sqrt{3} - \sqrt{2}) = 2 \sqrt{3}a - 2\varepsilon < |\vy_{i} - \vy_{j}|.$$ Thus $\vy_{i}$ and $\vy_{j}$ are not connected by an edge in $\Rng(\vphi(\varepsilon))$.
\end{proof}

\begin{lemma}\label{lem:edges_in_rng_from_outside}
	Let $\varepsilon > 0$ and $\lbrace u_1, u_2, \ldots \rbrace \subset \R^{3} \setminus B(o, 3(\sqrt{3}a + \varepsilon))$ be a locally finite point pattern such that the line segment $[u_1, u_2]$ hits the convex hull $\conv (\vphi(\varepsilon))$ of $\vphi(\varepsilon).$ Then, 
	\begin{itemize}
		\item[(i)] $u_1$ and $u_2$ are not connected by an edge in $\Rng(\vphi(\varepsilon) \cup \lbrace u_1, u_2, \ldots \rbrace),$ and
		\item[(ii)] if $u_1$ is connected with $\vy_{i}$ by an edge in $\Rng(\vphi(\varepsilon) \cup \lbrace u_1, u_2, \ldots \rbrace)$ for some $1 \leq i \leq 8,$ it is $|y_{9} - y_{i}| = \min \lbrace |y_{9} - y_{j}| : 1 \leq j \leq 8 \rbrace$.     
	\end{itemize} 
\end{lemma}
\begin{proof}
	Note that $\conv(\vphi(\varepsilon)) \subset B(o, \sqrt{3}a + \varepsilon).$ By assumption, there is a point $z_{0} \in [u_1, u_2] \cap \conv(\varphi(\varepsilon))$. Then,  $m = \min \lbrace |z_{0} - u_1|, |z_{0} - u_2| \rbrace$ is bounded from below and from above by
	
	$$ 2\sqrt{3}a + 2\varepsilon = 3(\sqrt{3}a+\varepsilon) - (\sqrt{3}a + \varepsilon) \leq m \leq |u_1 - u_2|/2.$$
		Since $|\vy_{1} - z_{0}| \leq |\vy_{1}| + |z_{0}| \leq 2 \sqrt{3}a + 2\varepsilon \leq m$, we see that 
		$$\max \lbrace |\vy_1 - u_1|, |\vy_1 - u_2|\rbrace \leq |\vy_1 - z_{0}| + \max \lbrace |z_0 - u_1|, |z_0 - u_2|\rbrace \leq |z_0 - u_1| + |z_0 - u_2| = |u_1 - u_2|.$$
		In particular,  $u_1$ and $u_2$ are not connected by an edge in $\Rng(\vphi(\varepsilon) \cup \lbrace u_1, u_2, \ldots \rbrace)$, thereby proving (i). To prove (ii), assume that there exists $j \neq i$ such that $|u_1 - \vy_{j}| < |u_1 - \vy_{i}|$ and $u_1$ is connected with $\vy_{i}$ by an edge in $\Rng(\vphi(\varepsilon) \cup \lbrace u_1, u_2 \rbrace).$ Then,
	$\max \lbrace |u_1 - \vy_{j}|, |\vy_{i} - \vy_{j}| \rbrace = |u_1 - \vy_{j}| < |u_1 - \vy_{i}|,$ which leads to a contradiction.
\end{proof}

After rescaling the Poisson point processes, we may assume that $\va=1$ in the following. Then, the points in $\vphi$ are the vertices of a cube with side length $2$.  Moreover, we also assume that
\begin{equation}\label{eq:condition_eps}
\varepsilon = (2 - \vr_1)/3 < (\sqrt{3} - \sqrt{2})/2,
\end{equation}
which is possible after increasing $\vr_1$ and thereby decreasing $\vr_2 = 4 - \vr_1.$
\begin{figure}[!htpb]
    \centering
    \begin{tikzpicture}[scale=1.5]
        \draw (-1,-1)--(1,-1)--(1,1)--(-1,1)--cycle;

        \draw (-1.83,-1.83)--(1.83,-1.83)--(1.83,1.83)--(-1.83,1.83)--cycle;
        
        \coordinate[label=-90: {$x_1$}] (v) at (-1,-1);
        \coordinate[label=-90: {$x_2$}] (v) at (1,-1);
        \coordinate[label=90: {$x_3$}] (v) at (1,1);
        \coordinate[label=90: {$x_4$}] (v) at (-1,1);

        \coordinate[label=-90: {$\widetilde{x}_1$}] (v) at (-1.83,-1.83);
        \coordinate[label=-90: {$\widetilde{x}_2$}] (v) at (1.83,-1.83);
        \coordinate[label=90: {$\widetilde{x}_3$}] (v) at (1.83,1.83);
        \coordinate[label=90: {$\widetilde{x}_4$}] (v) at (-1.83,1.83);

        \draw[dashed] (0,0) circle (2cm);

        \draw[thick, decorate,decoration=brace] (1,-1)--(0, -1);
        \coordinate[label=-90: \small{1}] (v) at (0.5, -1);

        \draw[thick, decorate,decoration={brace, aspect=0.60}] (1.83, 1.83)--(0, 1.83);
        \coordinate[label=-90: \small{$2\sqrt{2} - 1$}] (v) at (0.71, 1.83);

        \draw[thick, decorate,decoration={brace, aspect=0.40}] (2,-0)--(0, -0);
        \coordinate[label=-90: \small{2}] (v) at (1.2, -0);
    \end{tikzpicture}
    \caption{Projection of $\varphi$ and $\lbrace \widetilde{x}_1, \ldots, \widetilde{x}_8 \rbrace$ onto the plane orthogonal to $e_3$}
    \label{shieldFig}
\end{figure}
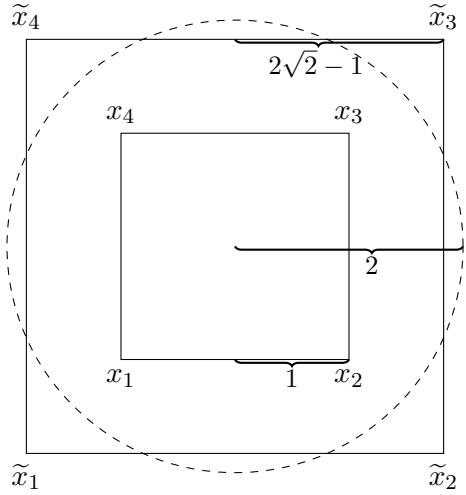
In the following, we consider the point pattern $\phi=\vphi \cup 6(\sqrt{3}c + \varepsilon) \vphi$, i.e. $\phi=\lbrace x_1, \ldots, x_{16} \rbrace,$ where $x_{i+8} = 6(\sqrt{3}c + \varepsilon) x_i$ for each $1 \leq i \leq 8.$ Analogously, we define $\phi(\varepsilon)=\lbrace y_1, \ldots, y_{16}\rbrace$, where $y_i \in B(x_i, \varepsilon)$ for each $1 \leq i \leq 16.$ Moreover, let $\widetilde{\phi} = \lbrace \widetilde{\vx}_{1}, \ldots, \widetilde{\vx}_{16} \rbrace$ and $\widetilde{\phi}(\varepsilon) = \lbrace \widetilde{\vy}_{1}, \ldots, \widetilde{\vy}_{16} \rbrace,$ where $\widetilde{\vx}_{i} = c \vx_{i}$ with $c = 2\sqrt{2} - 1$ and $\widetilde{\vy}_{i} \in B(\widetilde{\vx}_{i}, \varepsilon)$. Figure~\ref{shieldFig} shows the geometric relation between $\phi$ and $\widetilde{\phi}$. Note that  Lemmas~\ref{lem:edges_in_rng_on_cube} and~\ref{lem:edges_in_rng_from_outside} ensure that $\Rng(\phi)$ and $\Rng(\widetilde{\phi})$ exhibit the same topology. Let $$\eta = \lbrace z \in \R^{3}: \dist(z, \Rng(\phi)) \leq \dist(z, \Rng(\widetilde{\phi})) \rbrace$$ and 
$$\eta_{ \varepsilon} = \lbrace z \in \R^{3}: \dist(z, \Rng(\phi(\varepsilon))) \leq \dist(z, \Rng(\widetilde{\phi}(\varepsilon))) \rbrace.$$
Before we formulate Lemmas~\ref{lem:notin_ball} and~\ref{lem:inclusion_of_ball}, which are the basis for showing that $o \not \in \vPsi_{\vr_2}(\vXi)$ and $o \in \vXi^{\ominus \vr_1}$, respectively, we establish a result relating $\vPsi_r(\eta)$ with $\vPsi_{r+ 2\varepsilon}(\eta_\varepsilon)$. 

\begin{lemma}\label{lem:HausdorffDistance}
	Let $r, \varepsilon > 0$. Then,     $\vPsi_{r + 2 \varepsilon}(\eta_{ \varepsilon}) \subset \vPsi_{r}(\eta) \oplus B(o,2 \varepsilon).$
\end{lemma}
\begin{proof}
	Since the asserted inclusion is equivalent to $\eta_{ \varepsilon}^{\ominus r + 2\varepsilon} \oplus B(o,r + 2 \varepsilon)\subset \eta^{\ominus r } \oplus B(o,r + 2 \varepsilon),$
	it suffices to show that $\eta_{ \varepsilon}^{\ominus 2\varepsilon} \subset \eta$. Let $z \in \eta_{ \varepsilon}^{\ominus 2\varepsilon}$ be arbitrary. In particular, $$\dist(z, \Rng(\phi(\varepsilon))) + \varepsilon \le \dist(z, \Rng(\widetilde{\phi}(\varepsilon))) - \varepsilon,$$ so that
	$$\dist(z, \Rng(\phi)) \le \dist(z, \Rng(\phi(\varepsilon))) + \varepsilon \le \dist(z, \Rng(\widetilde{\phi}(\varepsilon))) - \varepsilon \le \dist(z, \Rng(\widetilde{\phi})),$$
	as required.
\end{proof}

\begin{lemma}\label{lem:notin_ball}
	Let $ r \geq 2 + 3\varepsilon$. Then, $o \notin \vPsi_{r}(\eta_\varepsilon).$
\end{lemma}
\begin{proof}
	At first, we show that $o \notin \vPsi_{r}(\eta) \oplus B(o, 2\varepsilon)$ for each $r \geq 2 + \varepsilon.$ Then, by Lemma~\ref{lem:HausdorffDistance}, it holds that $o \notin \vPsi_{r}(\eta_{\varepsilon})$ for each $r \geq 2 + 3 \varepsilon.$ Note that $\vPsi$ is anti-extensive \citep{matheron.1975}, i.e. $\vPsi_{r}(\eta) \subset \vPsi_{s}(\eta)$ for each $r \geq s.$ Thus it is sufficient to prove $o \notin \vPsi_{2 + \varepsilon}(\eta) \oplus B(o, 2\varepsilon),$ which is equivalent to $B(o,2 + 3 \varepsilon) \subset \eta^{c} \oplus B(o,2 + \varepsilon)$. Let now $z=(z_1, z_2, z_3) \in B(o,2 + 3 \varepsilon).$ To prove $z \in \eta^{c} \oplus B(o,2 + 2\varepsilon),$ we assume without loss of generality that $z_3 \geq z_2 \geq z_1 \geq 0.$ Furthermore, we treat the following two cases separately
	\begin{enumerate}
			\item[(i)] $z_3 \leq c$ and
			\item[(ii)] $z_3 > c.$
	\end{enumerate}
	Recall that $c = 2 \sqrt{2} - 1$. For case (i), consider the point $p_0 = (0, \sqrt{2}, \sqrt{2}).$ As
\begin{equation*}
\dist(\Rng(\phi), p_0) = \sqrt{2} (\sqrt{2} - 1) = \dist(\Rng(\widetilde{\phi}), p_0),
\end{equation*}  we obtain $p_0 \in \partial \eta.$ Additionally, 
	\begin{align*}
		|z - p_0| &= \sqrt{z_1^2 + (\sqrt{2} - z_2)^2 + (\sqrt{2} - z_3)^2} \\ 
		& \leq \max \left\lbrace \max_{z_3 \leq \sqrt{2}} \sqrt{z_1^2 + (\sqrt{2} - z_2)^2 + (\sqrt{2} - z_3)^2}, \max_{z_3 > \sqrt{2}} \sqrt{z_1^2 + (\sqrt{2} - z_2)^2 + (\sqrt{2} - z_3)^2} \right\rbrace
		\\ &\leq \max \left\lbrace \sqrt{z_1^2 + 2(\sqrt{2} - z_1)^2}, \sqrt{(2\sqrt{2} - 1)^2 + 2(\sqrt{2} - 1)^2} \right\rbrace \leq 2,
	\end{align*}
	which proves $z \in \eta^c \oplus B(o,2 + \varepsilon)$ in case (i). Next, we prove $z \in\eta^c \oplus B(o,2 + \varepsilon)$ for case~(ii). Note that in case (ii) it holds $\max \lbrace z_1, z_2 \rbrace \leq c$. Otherwise, using \eqref{eq:condition_eps} we would obtain $|z| \geq \sqrt{2} c = 4 - \sqrt{2} > 2 + 3(\sqrt{3} - \sqrt{2})/2 \geq 2 + 3\varepsilon,$ which contradicts $z \in B(o,2 + 3 \varepsilon).$ We consider the set of points $P = \lbrace p_1, p_2, p_3, p_4 \rbrace,$ where $p_1 = (c, c, c)$, $p_2=(c, 0, c)$, $p_3=(0, 0, 2\sqrt{2})$ and $p_4 = (0, c, c),$ see Figure~\ref{fig:construction_lemma29}.  
	Thus, $P \cap \mathring{\eta} = \emptyset$ as
	\begin{enumerate}
		    	\item[(i$^{\prime}$)] $\dist(\Rng(\phi), p_1) = 0 = \dist(\Rng(\widetilde{\phi}), p_1)$,
				\item[(ii$^{\prime}$)] $\dist(\Rng(\phi), p_i) = 2\sqrt{2} (\sqrt{2} - 1) > 0 =  \dist(\Rng(\widetilde{\phi}), p_i),$ for $i \in \lbrace 2,4 \rbrace,$
				\item[(iii$^{\prime}$)] $\dist(\Rng(\phi), p_3) = \sqrt{(2\sqrt{2} - 1)^2 + 1} =  \dist(\Rng(\widetilde{\phi}), p_3).$
	\end{enumerate}
	Note that the edge connecting $(1,1,1)$ with $6(\sqrt{3}c + \varepsilon) \cdot (1,1,1)$ goes through $p_1$ and thus $\dist(\Rng(\phi), p_1) = 0$. It is $p_1 \notin \mathring{\eta}$ since any neighborhood of $p_1$ is intersected by three edges of $\Rng(\widetilde{\phi}).$ Moreover, consider the set
	$ P' = \conv \lbrace p_1, p_2, p'_3, p_4 \rbrace \oplus (\lbrace 0 \rbrace \times \lbrace 0 \rbrace \times [0,1]),$
	where $p'_3 = (0, 0, c).$ For arbitrary $z' \in P' \cap (\R^{2} \times \lbrace c \rbrace)$ it holds 
		$$ \min_{1 \leq i \leq 4} |p_i - z'| \leq |p_3 - (c/2, c/2, c)| = \sqrt{c^2/2 + 1} < \sqrt{3}.$$
	Thus
		$$ \min_{1 \leq i \leq 4} |p_i - z'| \leq \sqrt{1 + 3} =  2$$
		for arbitrary  $z' \in P',$ which proves
		$P' \subset P \oplus B(o,2).$ Since
		$ c < z_3 < 2 + 3\varepsilon \leq 2 + 3(\sqrt{3} - \sqrt{2})/2 \leq 2 \sqrt{2},$ it is $z \in P'$. Thus $z \in \eta^c \oplus B(o, 2 + \varepsilon)$ in case (ii), which means that $o \notin \vPsi_{r}(\eta) \oplus B(o, 2\varepsilon)$ is verified for each $r \geq 2 + \varepsilon$. 
\end{proof}

\begin{figure}
	\includegraphics[width=0.45\textwidth]{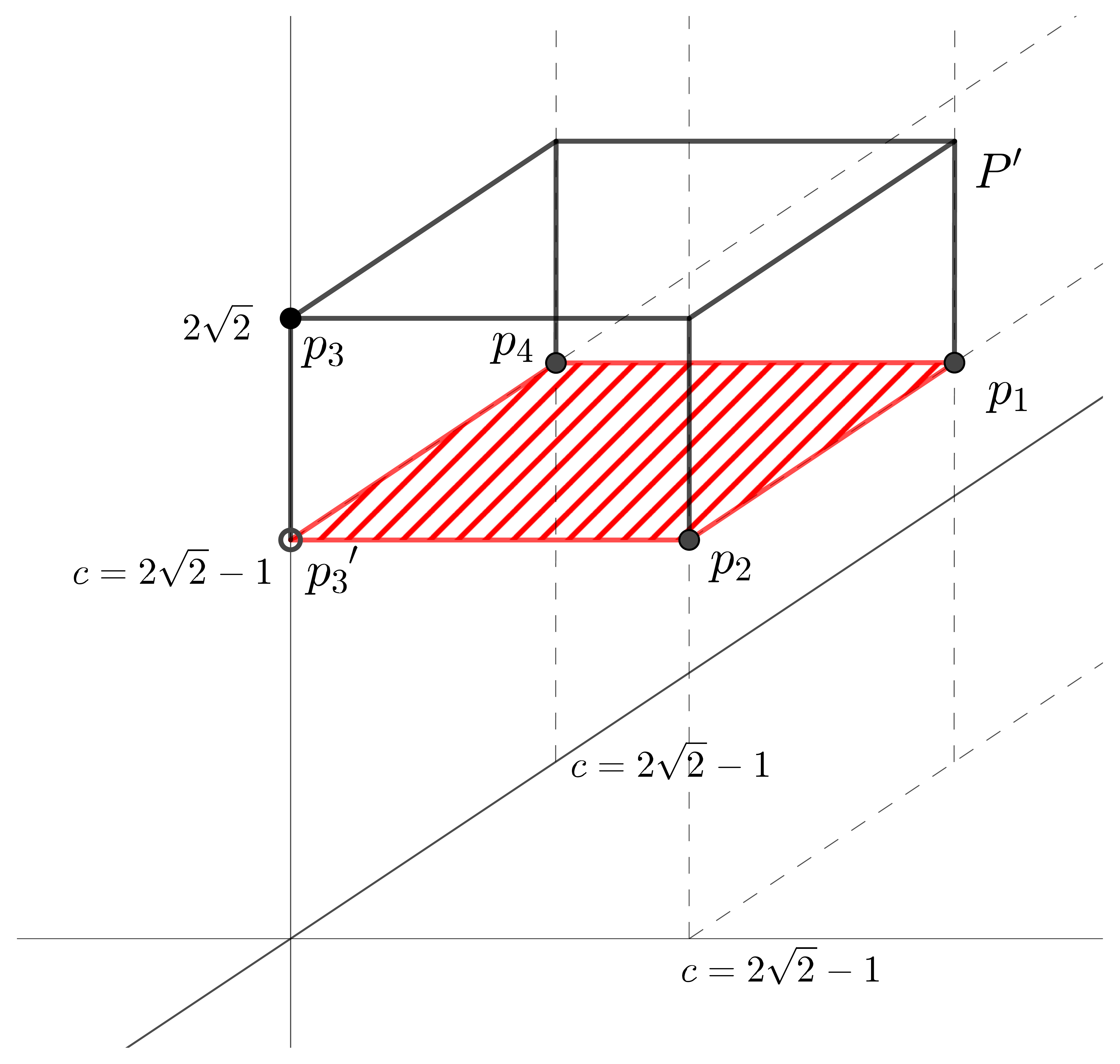}
	\caption{The sets $P$ and $P^{\prime}$ used for case (ii) in the proof of Lemma~\ref{lem:notin_ball}.}\label{fig:construction_lemma29}
\end{figure}

\begin{lemma}\label{lem:inclusion_of_ball}
	Let $\varepsilon > 0$. Then, $B(o,2 - 2\varepsilon) \subset \eta_{ \varepsilon}.$
\end{lemma}
\begin{proof}
	At first, we reduce the assertion to $B(o,2) \subset \eta$. Indeed, in order to deduce that $B(o,2 - 2\varepsilon) \subset \eta_{ \varepsilon}$, let now $z \in B(o,2 - 2\varepsilon)$. Then, using $B(o,2) \subset \eta$, we obtain
	\begin{align*}
		\dist(z, \Rng(\phi(\varepsilon))) \leq \dist(z, \Rng(\phi)) + \varepsilon \leq \dist(z, \Rng(\widetilde{\phi})) -  \varepsilon \leq \dist(z, \Rng(\widetilde{\phi}(\varepsilon))),
	\end{align*}
	which proves the assertion. To prove that $B(o, 2) \subset \eta$, let now $z = (z_{1}, z_{2}, z_{3}) \in B(o, 2)$ be arbitrary. We may assume without loss of generality that $z_3 \geq z_2 \geq z_1 \geq 0.$ Then, the distance between $z$ and $\Rng(\varphi)$ is minimized at the edge $[(-1,1,1),(1,1,1)].$ We treat the following two cases separately: 
	\begin{itemize}
		\item[(i)] $z_1 < 1$ and
		\item[(ii)] $z_1 \geq 1.$
	\end{itemize}
	In the first case the distance from $z$ to $\Rng(\varphi)$ is minimized at the interior of the edge of $\Rng(\varphi)$, whereas in the second case it is minimized at a vertex of $\Rng(\varphi).$ In case (i), we obtain
	$$\dist(z, \Rng(\phi)) \leq  \sqrt{\sum_{2 \le i \le 3}(z_{i} - 1)^2}\quad \text{ and }\quad \dist(z, \Rng(\widetilde{\phi})) =  \sqrt{\sum_{2 \le i \le 3}(z_{i} - c)^2},$$
	so that $z \in \eta$ follows from
	$z_{2} + z_{3} \leq 2\sqrt{2}.$
	 In particular, using the method of Lagrange multipliers,
	\begin{equation}\label{eq:i}
	\max \lbrace z_{2} + z_{3}:\, z \in B(o,2)\rbrace = \sqrt{8 - 2z_{1}^{2}} \leq 2\sqrt{2},
	\end{equation}
	so that $z \in \eta$. To deal with the second case, note that it is sufficient to prove the claim for $z \in \partial B(o,2).$ The result can be transferred to $z'$ with $|z'| < 2$ as
	$$\dist(z',\Rng(\widetilde{\phi})) = |z''-z'|+\dist(z'',\Rng(\widetilde{\phi})) \geq |z''-z'|+\dist(z'',\Rng(\phi)) \geq \dist(z',\Rng(\phi))$$
	for some $z'' \in \partial B(o,2).$
	Moreover, we may assume that  $z_{1} < 2/\sqrt{3} < c$. Otherwise, $|z| > \sqrt{3} (2/\sqrt{3}) > 2$ would contradict $|z| \le 2$. Then,
	$$\dist(z, \Rng(\phi)) \leq \sqrt{\sum_{1 \le i \le 3}(z_{i} - 1)^2} \quad\text{ and }\quad \dist(z, \Rng(\widetilde{\phi})) = \sqrt{\sum_{2 \le i \le 3}(z_{i} - c)^2},$$
	so that  $z \in \eta$ follows from
	\begin{equation*}\label{eq:ii}
		\frac{(z_1 - 1)^2}{c-1} + 2(z_2 + z_3) \leq 4\sqrt{2}.
	\end{equation*}
	Proceeding analogously to~\eqref{eq:i}, we see that it remains to verify
	$$\frac{(z_1 - 1)^2}{c-1} + 2 \sqrt{8 - 2z_1^2} \leq 4\sqrt{2}.$$
	As $1 < z_{1} < 2/\sqrt{3}$, the left-hand side is at most 
	$$\frac{(2/\sqrt{3} - 1)^2}{2\sqrt{2} - 2} + 2 \sqrt{6},$$
	which is smaller than $4\sqrt{2}$.
\end{proof}

\begin{proof}[Proof of Theorem \ref{thm:granulometry_decreasing}]
	We construct a specific event $\vA \in \A$ that occurs with positive probability and implies that $o \in \vXi^{\ominus \vr_1} \setminus \vPsi_{\vr_2}(\vXi)$.    To achieve this goal, we consider configurations where $\vX^{(1)}$ has points close to $\phi$ whereas $\vX^{(2)}, \ldots, \vX^{(k)}$ have points close to $\widetilde{\phi}$. In order to avoid possible disturbances coming from edges of the RNG to other points of $\vX^{(2)}, \ldots, \vX^{(k)}$, it is important that Poisson points are located in the neighborhood of $x_9, \ldots, x_{16}$ and $\widetilde{x}_9, \ldots, \widetilde{x}_{16}$. More precisely, we put     \begin{align*}
		\vA_1 = \bigcap_{\substack{1 \le i \le 16 \\ 2 \le k' \le k}}\left\lbrace \# (\vX^{(1)} \cap B(\vx_{i}, \varepsilon)) = \# (\vX^{(k')} \cap B(\widetilde{\vx}_{i}, \varepsilon)) = 1 \right\rbrace.
	\end{align*}
	In order to ensure that these points give rise to the desired configurations, we also assume that there are no further points in a sufficiently large neighborhood. More precisely, we put
	$$   \vA_2 = \bigcap_{1 \le k' \le k} \Big \lbrace \vX^{(k^{\prime})} \cap B(o,  54c^2 + 36 c \varepsilon) \subset  \bigcup_{1 \le i\le 16} \left(		B(\vx_{i}, \varepsilon) \cup B(\widetilde{\vx}_{i}, \varepsilon) \right) \Big\rbrace,
	$$
	and set $\vA = \vA_1 \cap \vA_2$. Since the family $\{\vX^{(i)}\}_{1 \le i \le 8}$ consists of independent Poisson point processes, we conclude that $A$ has positive probability. It remains to show that under the event $A$ we have that $o \in \vXi^{\ominus \vr_1} \setminus \vPsi_{\vr_2}(\vXi)$. Note that in each realization contained in $\vA$ there are 16 points of $\vX^{(1)}$ and 16 points of $\vX^{(2)}, \ldots, \vX^{(k)}$ contained in $B(o, 54c^2 + 36 c \varepsilon)$. We denote the points of $\vX^{(1)}$ in $\cup_{1 \leq i \leq 8}(B(\vx_{i}, \varepsilon) \cup B(\vx_{i+8}, \varepsilon))$ by $\vy_{1}, \ldots, \vy_{16}$ and the points of $\vX^{(2)}, \ldots, \vX^{(k)}$ in $\cup_{1 \leq i \leq 8}(B(\widetilde{\vx_{i}}, \varepsilon) \cup B(\widetilde{\vx}_{i+8}, \varepsilon))$ by $\widetilde{\vy}_{k^\prime, 1}, \ldots, \widetilde{\vy}_{k^\prime, 16},$ for all $2 \le k^{\prime} \le k$ where $\vy_1, \ldots, \vy_8$ and $\widetilde{\vy}_{k^\prime, 1}, \ldots, \widetilde{\vy}_{k^\prime, 8}$ form the inner cubes. The relative neighborhood graphs restricted to the points within $B(o,  54c^2 + 36 c \varepsilon)$ exhibit the same topology for all realizations of $A$ by Lemmas~\ref{lem:edges_in_rng_on_cube} and~\ref{lem:edges_in_rng_from_outside}. 
	It remains to show that under the event $A$ we have $o \in \vXi^{\ominus \vr_1} \setminus \vPsi_{\vr_2}(\vXi)$. By Lemma~\ref{lem:edges_in_rng_from_outside}, under the event $A_1$ only those points of $\vX^{(1)}, \ldots, \vX^{(k)}$ located within $B(o,  54c^2 + 36 c \varepsilon)$ are relevant for the event $\lbrace o \in \vXi^{\ominus \vr_1} \setminus \vPsi_{\vr_2}(\vXi)\rbrace$. Furthermore, as indicated in the beginning of the proof, this event is not influenced by edges having a vertex in $\cup_{9 \le i\le 16}B(\vx_{i}, \varepsilon)$. Since $r_1 < 2 - \varepsilon$, we may therefore apply Lemma~\ref{lem:inclusion_of_ball} to conclude that $B(o, \vr_1) \subset \vXi$. Since $r_2 = 2 + 3 \varepsilon,$ we obtain $o \notin \vPsi_{\vr_2}(\vXi)$ by Lemma~\ref{lem:notin_ball}. 
\end{proof}

\section{Geometry of the Voronoi mollification}\label{app:condition_c}
As explained in Section~\ref{constRngSec}, the only missing ingredient in the proof of Lemma~\ref{highProbLem} is the connectivity condition {\bf (C')}. This is surprisingly complicated and requires a detailed analysis of the cell geometry in the Voronoi mollification.  To make this precise, let $\vG=\{\vI_i\}_{i\ge1}$ be a locally finite system of line segments in $\R^3$. We assume that distinct line segments are either disjoint or have one common endpoint. We also let $\vV=\cup_{i\ge1}(\vI_i\setminus \mathring{\vI}_i)$ denote the vertex set of $\vG$, where $\mathring{\vI}_i$ denotes the relative interior of the line segment $I_i$. Now, using a similar construction as in~\citep{chew.1998}, we introduce a \emph{line-segment Voronoi tessellation} associated with the vertices and edges of the graph $\vG$. A first idea would be to define the cell associated to a line segment $\vI$ as the family of all points of $\R^3$ whose closest point on $\vG$ is located on $\vI$. However, this would mean that cells associated to line segments meeting at a common vertex share some non-empty interior. This problem is illustrated in Figure~\ref{cellFig}. To preserve consistency with the intuitive concept of a tessellation, we provide a different definition that treats points and line segments separately. 

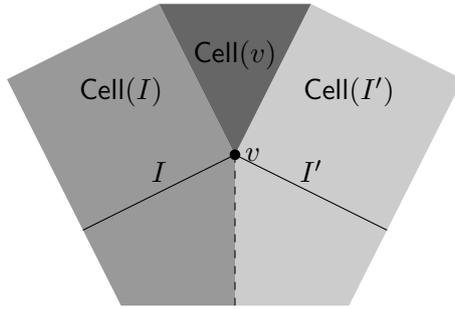
\begin{figure}[!htpb]
	\centering
	\begin{tikzpicture}
		\fill[black!20!white] (0,0)--(0,-2)--(1.5,-2)--(3,1)--(1,2);
		\fill[black!40!white] (0,0)--(0,-2)--(-1.5,-2)--(-3,1)--(-1,2);
		\fill[black!60!white] (-1,2)--(0,0)--(1,2);
		\draw (-2,-1)--(0,0);
		\draw (0,0)--(2,-1);
		\draw[dashed] (0,0)--(0,-2);
		\fill (0,0) circle (2pt);
		\coordinate[label=90: {$\vCell(v)$}] (v) at (0,1);
		\coordinate[label=90: {$\vCell(I)$}] (v) at (-1.5,0.5);
		\coordinate[label=90: {$\vCell(I')$}] (v) at (1.5,0.5);
		\coordinate[label=00: {$v$}] (o) at (0,0);
		\coordinate[label=90: {$I$}] (o) at (-1,-0.5);
		\coordinate[label=90: {$I'$}] (o) at (1,-0.5);
	\end{tikzpicture}
	\caption{Two-dimensional illustration of line-segment Voronoi cells (cut-out)}
	\label{cellFig}
\end{figure}

More precisely, if $v\in\vV$ we let 
$\vCell(v)=\{x\in\R^3:\, \dist(x,v)\le\dist(x,\cup_{i\ge1}I_i)\}$
denote the family of all points whose closest point on $\cup_{i\ge1}I_i$ is given by $v$. 
Similarly, if $\vI\in\vG$ we let 
$$\vCell(\vI)=\overline{\{x\in\R^3:\, \dist(x,\vI)<\dist(x,(\cup_{i\ge}\vI_i)\setminus\mathring{\vI})\}}$$

denote the closure of the open cell of points that are closer to $\vI$ than to any other edge of $\vG$. The two types of cells are illustrated in Figure~\ref{cellFig} and are referred to as \emph{point cell} and \emph{line-segment cell} in the following.

The geometry of the line-segment Voronoi tessellation is substantially more intricate than the geometry of a standard Voronoi tessellation induced by a locally finite set of points. Indeed, whereas all cells in the standard Voronoi tessellation are convex sets with planar boundaries, in the line-segment model also quadratic surfaces may appear as boundary shapes. More precisely, the following three possibilities can be observed:
\begin{enumerate}
\item \emph{a plane}, as boundary between two cells associated either with two vertices or with two coplanar line segments,
\item \emph{a parabolic cylinder}, as boundary between a cell associated with a vertex and a cell associated with a line segment, and
\item \emph{a hyperbolic paraboloid}, as boundary between two cells associated with skew line segments.
\end{enumerate}

For the standard Voronoi tessellation in $\R^3$ with generating points in normal position, it is a classical result that any four cells meet in at most one point~\citep{schneider.2008}. Indeed, as the locus of all points having equal distance to two points is a plane, we see that the locus of points that have equal distance to three points is a line. This implies the result, since any two distinct lines in $\R^3$ can intersect in at most $1$ point. 

This string of argumentation does not extend immediately to line-segment Voronoi tessellations, as not only planes but also smooth quadrics can appear as cell boundaries. Therefore, also the intersection of three cells is no longer a line but can be a rather complicated algebraic curve. Nevertheless, according to Sard's theorem~\citep[Chapter 6]{lee.2013} one should expect the intersection curves to be almost surely non-singular, so that any two of them can only intersect in finitely many points. Fortunately, since we are only dealing with quadrics and not general curves, we can make this idea rigorous without resorting to the general form of Sard's theorem.
Let $\Vor(\Rng(\vX\vda{1})\cup\Rng(\vX\vda{2}))$ denote the line-segment Voronoi tessellation. To simplify terminology, we say that a point $x\in\R^3$ is a \emph{quad point} if it is contained in the intersection of four distinct cells of $\Vor(\Rng(\vX\vda{1})\cup\Rng(\vX\vda{2}))$. 

\begin{lemma}
\label{intSectLem}
With probability $1$, the family of quad points is locally finite.
\end{lemma}
\begin{proof}

	We assert that with probability $1$, the intersection of any three cells is given by a non-singular algebraic curve. Once this assertion is shown, basic algebraic geometry implies that quad points are locally finite. Indeed, by~\citep[Proposition 5.5]{liu.2008} the intersection of any two non-identical non-singular curves is of dimension $0$ and therefore consists of a finite union of points~\citep[Proposition 4.9, Lemma 5.11]{liu.2008}. Moreover, with probability $1$, the four curves associated to the intersection of triplets of cells are non-identical. We show this only for the case of line-segment cells, since the proof for point cells is similar but easier. If two curves were identical, then one of the lines would be tangent to the ball touching the other three lines around the closest point to the origin on one of these curves. But this is an event of probability 0. It remains to show that with probability $1$, the intersection set of any three cells is given by a non-singular algebraic curve. In fact, for the case that the three cells are deterministic line-segment cells, in~\citep{everett.2009} a substantially more specific algebraic description of the intersection curve is given. However, as we can neglect configurations of probability $0$, we only need a small part of the proof in~\citep{everett.2009}. Hence, we are able to deal with point cells and line-segment cells using the same argument. For any two quadric surfaces $Q$, $Q'$ there exists a \emph{characteristic polynomial} $D_{Q,Q'}(\lambda)$ with the property that $Q$ and $Q'$ intersect in a non-singular curve if and only if $D_{Q,Q'}(\lambda)$ does not admit multiple roots~\citep[p.~100]{everett.2009}. In particular, $Q$ and $Q'$ intersect in a non-singular curve if and only if the associated discriminant $\Delta_{Q,Q'}$ is non-zero. If $Q$ and $Q'$ describe cell boundaries, then the coefficients of the characteristic polynomial are polynomials in the coordinates of points in $X^{(1)} \cup X^{(2)}$. In particular, also $\Delta_{Q,Q'}$ is a polynomial in the coordinates of the Poisson point process. Intuitively speaking, it is clear that coordinates chosen at random do not satisfy any fixed non-trivial algebraic equation. Nevertheless, for the convenience of the reader, we provide some details. First, we observe that $\Delta_{Q,Q'}$  is not the zero polynomial. Indeed, this would imply that the intersection between any two quadrics can \emph{never} be a non-singular curve. In particular, we can fix one of the Poisson coordinates -- call it $u$ -- and consider $\Delta_{Q,Q'}=\Delta_{Q,Q'}(u)$ as polynomial in $u$ with coefficients being polynomial expressions in the remaining Poisson coordinates. By induction on the polynomial degree, we see that when evaluated at Poisson points, these coefficients are almost surely non-zero. In particular, for almost every choice of the remaining Poisson coordinates the equation $\Delta_{Q,Q'}(u) = 0$ has only finitely many roots in $u$. Therefore, also $\Delta_{Q,Q'}$ is almost surely non-zero.
\end{proof}

Lemma~\ref{intSectLem} implies that the family of points that are simultaneously close to four cells of $\Vor(\Rng(\vX\vda{1})\cup\Rng(\vX\vda{2}))$ decomposes into connected components of small diameters. More precisely, for $r>0$ we let $E^{\ms{quad}}_r$ denote the event that in $\Vor(\Rng(\vX\vda{1})\cup\Rng(\vX\vda{2}))$ there exist pairwise distinct cells $C_1,C_2,C_3$ and $C_4$ such that the cells are associated with points or line segments intersecting $\vQ_{3\vL}$  and such that the diameter of each connected component of the set $\cap_{i\le4}(C_i \oplus \vB(o,r))$ is larger than $1$.

\begin{corollary}
	\label{intSectCor}
	Let $L\ge1$ be arbitrary. Then, $\lim_{r\to0} \P(E^{\ms{quad}}_r) = 0$.
\end{corollary}
\begin{proof}
	The decreasing limit of compact sets $\cap_{r>0}\cap_{i\le4}(C_i \oplus \vB(o,r))$ equals $\cap_{i\le4}C_i$, which is finite by Lemma~\ref{intSectLem}. Now, set
	$$r'=\tfrac14\min\Big\{1,\min_{\substack{x,y\in\cap_{i\le4}C_i\\ x\ne y}}|x-y|\Big\}.$$
	In particular, we may choose a sufficiently small $r>0$ such that 
	$$\cap_{i\le4}(C_i \oplus \vB(o,r))\subset(\cap_{i\le4} C_i) \oplus \vB(o,r').$$
	By the choice of $r'$ the connected components of the right-hand side are of diameter at most $1$, as required.
\end{proof}

As mentioned above, the difficulty in verifying condition {\bf (C')} lies in the geometric complexity of the line-segment Voronoi tessellation when compared to the standard Voronoi tessellation. As an important preliminary step, we note that if for $x\in\vXir$ there exists $x'\in\Rng(\vX\vda{1})$ such that for all $y\in\vB(x,r)$ the distance from $y$ to $x'$ is at most the distance from $y$ to $\Rng(\vX\vda{2})$, then we may conclude as in the standard Voronoi tessellation.
\begin{lemma}
	\label{standardConnLem}
	Let $x\in\vXir$ and $x'\in\Rng(\vX\vda{1})$ be such that 
	\begin{align}
		\label{standardConnEq}
		|x-x'|+2r\le\dist(x,\Rng(\vX\vda{2})).
	\end{align}
	Then, the line segment $[x,x']$ is contained in $\vXir$.
\end{lemma}
\begin{proof}
	Let $z\in\Rng(\vX\vda{2})$, $t\in[0,1]$ and $y\in\vB(x,r)$ be arbitrary. Then,
	$$|y+t(x'-x)-x'|\le r+|x'-x|-t|x'-x|\le |x-z|-r-t|x-x'|\le |y+t(x'-x)-z|,$$
	as required.
\end{proof}

Now, we can complete the verification of property {\bf (C')}.
\begin{proof}[Proof of {\bf (C')}]
	First, let $L \geq 1$ and
	$$r_0=\tfrac14\min\big\{1,\dist\big(\Rng(\vX\vda{1})\cap Q_{2L},\Rng(\vX\vda{2})\big)\big\}.$$
Now, fix $r\in(0,r_0^2/(20L))$. In particular, $\Rng(\vX\vda{1})\cap\vQ_{2\vL}\subset\vXir$. In the following, let $x\in\vXir.$ According to {\bf (O')}, see Section \ref{constRngSec}, we may assume that $x\in\vXir \cap \vQ_L$. Roughly speaking, we first move the point $x$ in a carefully chosen direction $v \in \partial B(o, 1)$ by a distance of $Kr$ for suitably chosen $K > 0$ to arrive at $x^{*}=x+Krv$. Then, we prove (i) $[x,x^{*}]\subset\vXir$ and (ii) $[x^{*},\pi_1(x^*)] \subset \vXir$, where $\pi_i(x)$ denotes a closest point to $x$ on the graph $\Rng(\vX\vda{i})$. To choose the good direction $v$, we proceed as follows. For any fixed $K > 0$, we note that after decreasing the upper bound on $r$ if necessary,  by Corollary~\ref{intSectCor} we may assume that the ball $\vB(x,2Kr)$ intersects at most three cells of $\Vor(\Rng(\vX\vda{1})\cup\Rng(\vX\vda{2}))$. If the ball intersects at most one cell associated with $\Rng(\vX\vda{1})$, then we choose $$v = (x - \pi_1(x))/|x - \pi_1(x)|$$ so as to push $x$ closer to the cell. On the other hand, if the ball intersects at most one cell associated with $\Rng(\vX\vda{2})$, we choose $$v = - (x - \pi_2(x))/|x - \pi_2(x)|.$$ That is, the point is pushed away from that cell. We provide a detailed proof for the second case, noting that similar arguments apply in the first case. We put $K = 14000L^3r_0^{-3}$ and may assume that $x$ does not satisfy~\eqref{standardConnEq}, as otherwise we could conclude the proof immediately. In particular, 
    \begin{align}
        \label{lowPiBound}
        |x - \pi_1(x)|\ge |x - \pi_2(x)| - 2r.
    \end{align}
     For part (i), we assert that $y_t = y + tv\in\vXi$ for each $0 \le t \le Kr$ and $y \in \vB(x,r)$. To prove this assertion, we have to distinguish the cases depending on whether the considered cell is a point cell or a line-segment cell. In the line-segment case, the triangle inequality yields that 
	 \begin{align*}
	 	\dist(y_t,I) = \dist(y, I) + t \ge \dist(y, \Rng(\vX\vda{1})) + t \ge \dist(y_t, \Rng(\vX\vda{1})),
	 \end{align*}
	 	 where $I$ is the line segment inducing the cell associated with $\Rng(\vX\vda{2}).$  
     Next, assume that the considered cell is a point cell induced by $\pi_2(x)$. In this case, typically it holds $\dist(y_t,I) < \dist(y, I) + t$ and we can not use the same arguments as in the line-segment case. Then, 
     \begin{align}
         \label{projEq}
         \pi_2(y_t) = \pi_2(x)
     \end{align}
     for each $0 \le t \le Kr$ and $y \in \vB(x,r)$ because this cell is the only cell of the Voronoi tessellation on $\Rng(\vX\vda{2})$ that hits $B(x, 2Kr)$. For $y\in\vB(x,r)$ we put $\widetilde{y} = \pi_2(x)+y-x$ and combine~\eqref{lowPiBound} and~\eqref{projEq} to arrive at
	$$ |x - \pi_2(x)| \le |x - \pi_1(x)| + 2r \le |x - \pi_1(y)| + 2r \le |y - \pi_1(y)| + 3r,$$
	and 
	$$|y - \pi_1(y)| \le |y - \pi_2(y)| \le |x - \pi_2(x)| + r.$$
	Letting $\alpha_y$ denote the angle between $\pi_1(y) - y$ and $\widetilde{y} - y = \pi_2(x) - x$, we therefore obtain that 
	 \begin{align}
		 \label{alphaBoundEq}
		 1-\cos\alpha_y=\frac{|\wt{y} - \pi_1(y)|^2 - (|y - \pi_1(y)| - |x - \pi_2(x)|)^2}{2|y - \pi_1(y)| |x - \pi_2(x)|}		 \ge \frac{9(r_0^2-r^2)}{24L^2}\ge\frac{r_0^2}{3L^2}.
	 \end{align}
	 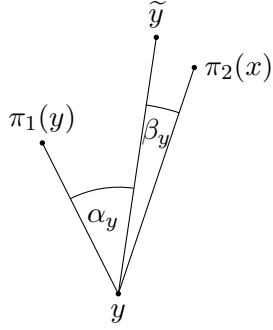
\begin{figure}[t]
	 	\centering
	 	\begin{tikzpicture}
	 	
	 	\coordinate[label=-90: {$y$}] (y) at (0,0);
	 	\coordinate[label=0: {$\pi_2(x)$}] (p2x) at (1,3);
	 	\coordinate[label=90: {$\wt{y}$}] (wty) at (0.5,3.4);
	 	\coordinate[label=90: {$\pi_1(y)$}] (p1x) at (-1,2);
	 	\coordinate[label=90: {$\beta_y$}] (by) at (0.5,1.8);
	 	\coordinate[label=90: {$\alpha_y$}] (by) at (-0.2,0.7);
	 	
	 	\draw (y)--(p2x);
	 	\draw (y)--(wty);
	 	\draw (y)--(p1x);
	 	
	 	\fill (y) circle (1pt);
	 	\fill (wty) circle (1pt);
	 	\fill (p2x) circle (1pt);
	 	\fill (p1x) circle (1pt);
	 	
	 	\draw (0.8,2.4) arc (71.6:82:2.5);
	 	\draw (0.2,1.4) arc (82:117:1.4);
	 	\end{tikzpicture}
	 	\caption{Configuration involving $y$, $\wt{y}$, $\pi_1(y)$ and $\pi_2(x)$}
	 	\label{alphayFig}
	 \end{figure}
	 Similarly, letting $\beta_y$ denote the angle between $\pi_2(x)-y$ and $\widetilde{y}-y=\pi_2(x)-x$, we obtain that 
	 \begin{align*}
		 1-\cos\beta_y=\frac{|\wt{y}-\pi_2(x)|^2-(|y-\pi_2(x)|-|x-\pi_2(x)|)^2}{2|y-\pi_2(x)||x-\pi_2(x)|} \le \frac{r^2}{4r_0^2},
	 \end{align*}
	 so that $\cos\beta_y\ge \cos\alpha_y$. Here we use 
	 $$2 |x-\pi_2(x)| \geq |\pi_1(x) - x| + |x - \pi_2(x)| \geq |\pi_1(x) - \pi_2(x)| \geq 4 r_0$$
	 by the definition of $r_0$. Analogously, one can show 
	 $|y -\pi_2(x)| \geq r_0.$ For an illustration of the geometric configuration, we refer the reader to Figure~\ref{alphayFig}. In particular, 
	\begin{align*}
		&|y_t-\pi_2(y_t)|-|y_t-\pi_1(y)| = \frac{|y_t-\pi_2(x)|^2-|y_t-\pi_1(y)|^2}{|y_t-\pi_2(x)| + |y_t-\pi_1(y)|}\\
			      &\quad=\frac{|y-\pi_2(x)|^2+t^2+2t|y-\pi_2(x)|\cos\beta_y-|y-\pi_1(y)|^2-t^2-2t|y-\pi_1(y)|\cos\alpha_y}{|y_t-\pi_2(x)| + |y_t-\pi_1(y)|}\\
		&\quad\ge(\cos\beta_y-\cos\alpha_y)\frac{2t|y-\pi_2(x)|}{|y_t-\pi_2(x)|+|y_t-\pi_1(x)|},
	\end{align*}
	which was shown above to be non-negative. For part (ii), we apply Lemma~\ref{standardConnLem} to reduce the assertion to $|x^{*} - \pi_2(x^*)| - |x^{*}-\pi_1(x^*)| \ge 2r$. Now using~\eqref{alphaBoundEq} for the angle $\alpha_x$ implies that
	\begin{align*}
		&|x^{*}-\pi_2(x^*)|-|x^{*}-\pi_1(x^*)|\\
		&\quad\ge|x^{*}-\pi_2(x)|-|x^{*}-\pi_1(x)|\\
		&\quad=\frac{|x^{*}-\pi_2(x)|^2-|x^{*}-\pi_1(x)|^2}{|x^{*}-\pi_2(x)|+|x^{*}-\pi_1(x)|}\\
		&\quad=\frac{(|x^{}-\pi_2(x)|+Kr)^2-(|x^{}-\pi_1(x)|+Kr)^2+2Kr|x^{}-\pi_1(x)|(1-\cos\alpha_x)}{|x^{*}-\pi_2(x)|+|x^{*}-\pi_1(x)|}\\
		&\quad\ge\frac{2Kr|x-\pi_1(x)|(1-\cos\alpha_x)}{|x^{*}-\pi_2(x)|+|x^{*}-\pi_1(x)|}\\
		&\quad\ge r\frac{Kr_0^3}{7000L^3},
	\end{align*}
	and by the choice of $K$, the last line is bounded below by $2r$, as required.
	\end{proof}

\bibliography{C:/Users/matthias/Bib/Literatur}
\bibliographystyle{rss}

\vspace{1cm}

Correspondence: Matthias Neumann, Institute of Stochastics, Ulm University, Helmholtzstra\ss e 18, 89069 Ulm, Germany, email: matthias.neumann@uni-ulm.de
\end{document}